\documentclass[11pt]{amsart}

\usepackage{amsthm,amssymb,amstext,amscd,amsfonts,amsbsy,amsrefs,amsxtra,latexsym,amsmath,xcolor,mathrsfs,fancybox,upgreek, soul,url}
\usepackage[english]{babel}
\usepackage[all,cmtip]{xy}
\usepackage[latin1]{inputenc}
\usepackage{cancel}
\usepackage{hyperref}
\usepackage{comment}
\usepackage{mdframed}
\usepackage{mathtools}
\usepackage{enumerate}
\usepackage{thmtools}
\usepackage{thm-restate}

\usepackage{chngcntr}
\usepackage{enumitem}
\usepackage{etoolbox}
\usepackage{todonotes}
\usepackage{footmisc}
\usepackage{bigints}
\usepackage{enumerate}
\usepackage{graphicx}
\usepackage{youngtab}
\usepackage{bm}
\usepackage{caption}
\usepackage{float}
\theoremstyle{plain}
\newtheorem{theorem}{Theorem}[section]
\newtheorem{corollary}[theorem]{Corollary}
\newtheorem{lemma}[theorem]{Lemma}
\newtheorem{proposition}[theorem]{Proposition}

\theoremstyle{definition}

\theoremstyle{remark}
\newtheorem*{remark}{Remark}
\newtheorem*{example}{Example}

\newtheorem*{multRem}{Two Remarks}

\newcommand{\Z}{\mathbb{Z}}
\newcommand{\Q}{\mathbb{Q}}
\newcommand{\GL}{\mathrm{GL}}

\newcommand{\F}{\mathbb{F}}

\newcommand{\HH}{\mathbb{H}}

\newcommand{\N}{\mathbb{N}}
\newcommand{\SL}{\operatorname{SL}}
\newcommand{\Mp}{\operatorname{Mp}}
\newcommand{\Cl}{\mathrm{Cl}}

\newcommand{\C}{\mathbb{C}}

\catcode`,\active

\catcode`\,12

\numberwithin{equation}{section}

\DeclareMathOperator{\ST}{ST}
\DeclareMathOperator{\Bat}{Bat}
\DeclareMathOperator{\hol}{hol}

\begin{document}

\title[Explicit Sato-Tate type distribution for a family of $K3$ surfaces]{Explicit Sato-Tate type distribution for a family of $K3$ surfaces}

\author{Hasan Saad}

\address{Department of Mathematics, University of Virginia, Charlottesville, VA 22904}
 \email{hs7gy@virginia.edu}


\begin{abstract}  

In the 1960's, Birch proved that the traces of Frobenius for elliptic curves taken at random over a large finite field is modeled by the semicircular distribution (i.e. the usual Sato-Tate for non-CM elliptic curves).  In analogy with Birch's result, a recent paper by Ono, the author, and Saikia proved that the limiting distribution of the normalized Frobenius traces $A_{\lambda}(p)$ of a certain family of $K3$ surfaces $X_\lambda$ with generic Picard rank $19$ is the $O(3)$ distribution. This distribution, which we denote by $\frac{1}{4\pi}f(t),$ is quite different from the semicircular distribution. It is supported on $[-3,3]$ and has vertical asymptotes at $t=\pm1.$ Here we make this result explicit. We prove that if $p\geq 5$ is prime and $-3\leq a<b\leq 3,$ then
$$
\left|\frac{\#\{\lambda\in\F_p :A_{\lambda}(p)\in[a,b]\}}{p}-\frac{1}{4\pi}\int_a^b f(t)dt\right|\leq \frac{110.84}{p^{1/4}}.
$$
As a consequence, we are able to determine when a finite field $\F_p$ is large enough for the discrete histograms to reach any given height near $t=\pm1.$ To obtain these results, we make use of the theory of Rankin-Cohen brackets in the theory of harmonic Maass forms.

\end{abstract}
\maketitle
\section{Introduction and statement of results}\label{SectionIntro}

If $E/\Q$ is an elliptic curve, then a theorem of Hasse implies, for every prime $p,$ that there exists $\theta_E(p)\in[0,\pi]$ such that
$$
\# E(\F_p)=p+1-2\sqrt{p}\cos(\theta_E(p)).
$$
Around 1960, Sato and Tate independently conjectured that if $E$ does not have complex multiplication and $0\leq a<b\leq \pi,$ then 
$$
\lim\limits_{N\to\infty}\frac{\#\left\{p\leq N : \theta_E(p)\in[a,b]\right\}}{N}=\frac{2}{\pi}\int_a^b\sin^2\theta d\theta.
$$
Under some conditions, the Sato-Tate Conjecture was famously proved by Clozel, Harris, Shepherd-Barron, and Taylor in three joint papers \cite{TaylorSatoTate1a},\cite{TaylorSatoTate1c}, and \cite{TaylorSatoTate1b}. An unconditional proof was obtained in 2011 by Barnet-Lamb, Geraghty, Harris, and Taylor \cite{TaylorSatoTate2}.

There are other arithmetic questions that have been considered about such angles. Indeed, almost immediately after the Sato-Tate conjecture was formulated, Birch \cite{birch} proved that the distribution of $\theta_E(p),$ where $E$ varies over all elliptic curves defined over $\F_p,$ is the Sato-Tate distribution as $p\to\infty.$  A recent paper by Murty and Prabhu \cite{murty} makes Birch's result effective by bounding the error for each finite field $\F_{p^r}.$ To make this precise, for $(c_1,c_2)\in\F_{p^r}^2$ such that $p\nmid 4c_1^3+27c_2^2,$ consider the elliptic curve
$$
E_{c_1,c_2}:\ \ y^2=x^3+c_1x+c_2.
$$
In this notation, if $-2<a<b<2$ and $r$ is a fixed positive integer, then
\begin{equation}\label{EqMurty}
	\frac{\#\left\{(c_1,c_2)\in\F_{p^r}^{2}, p\nmid (4c_1^3+27c_2^2), a_{c_1,c_2}(p^r)\in[a,b]\right\}}{p^{2r}}=\frac{1}{2\pi}\int_a^b \sqrt{4-t^2}dt+O_r(p^{7r/4}),
\end{equation}
where $a_{c_1,c_2}(p^r)=q+1-\#E_{c_1,c_2}(p^r).$

Here we study an analogous question for the $K3$ surfaces (see \cite{onoK3}) defined by 
$$
X_\lambda:\ \ \ s^2=xy(x+1)(y+1)(x+\lambda y),
$$
where $\lambda\in\F_p\setminus\{0,-1\}.$ The family $\{X_\lambda:\ \lambda\in\Q\setminus\{0,-1\}\}$ has several important properties. In particular, for all but finitely many $\lambda,$ the surfaces $X_\lambda$ have Picard rank $19.$ Moreover, Ahlgren, Ono, and Penniston proved \cite{onoK3} that if $p$ is an odd prime and $\lambda\in\F_p\setminus\{0,-1\},$ then the local zeta function of $X_{\lambda}$ at $p$ is
\begin{equation}\label{ZetaFunction}
Z\left(X_\lambda/\F_p, T\right)=\frac{1}{(1-T)(1-p^2T)(1-pT)^{19}(1-\gamma pT)(1-\gamma\pi_{\lambda,p}^2T)(1-\gamma\overline{\pi}_{\lambda,p}^2T)},
\end{equation}
where $\pi_{\lambda,p}$ and $\overline{\pi}_{\lambda,p}$ are the Frobenius eigenvalues of the elliptic curve $E^{\Cl}_{\frac{-1}{\lambda+1}}$ (see \ref{ClausenFamily}) and $\gamma\in\{\pm1\}.$
In light of (\ref{ZetaFunction}), the number of $\F_p$-rational points on $X_\lambda$ is approximately $p^2+19p+1.$ Therefore, it is natural to define $A_\lambda(p)$ by 
\begin{equation}
\# X_\lambda(\F_{p})=1+p^{2}+19p+p\cdot A_{\lambda}(p).
\end{equation}

We have that $A_{\lambda}(p)$ is in $[-3,3]$ for all $p,$ which turns out to be a consequence of Hasse's Theorem. An analogue of Birch's theorem would be to study the distribution of $A_{\lambda}(p),$ where $\lambda$ varies over $\F_p\setminus\{0,-1\},$ as $p\to\infty.$
In this direction, Ono, the author, and Saikia \cite{BatmanPaper} recently proved that the distribution of $A_{\lambda}(p),$ unlike the semicircular distribution, turns out to be the $O(3)$ distribution\footnote{The moments of this distribution arise \cite{pastur} as the moments of traces of the real orthogonal group $O(3).$}. To be precise, if $-3\leq a<b\leq 3,$ then
$$
\lim\limits_{p\to\infty}\frac{\#\left\{\lambda\in\F_{p} : A_\lambda(p)\in [a,b]\right\}}{p}=\frac{1}{4\pi}\int_a^b f(t)dt,
$$
where
\begin{equation}
f(t)=\begin{cases}
\sqrt{\frac{3-|t|}{1+|t|}} \ \ \ \ &\ \ \ \ {\text {\it if}}\  1<|t|<3, \\ \\
\sqrt{\frac{3-t}{1+t}}+\sqrt{\frac{3+t}{1-t}} &\ \ \ \  {\text{{\it if}}}\  |t|<1,\\ \\
0 &\ \ \ \ \text{otherwise}.
\end{cases}
\end{equation}
One striking feature is that the function $f(t)$ has vertical asymptotes at $t=\pm1.$
\begin{example}
For the prime $p=93283,$ we compare the histogram of the distribution of $A_\lambda(p),$ for $\lambda\in\F_{93283},$ with the limiting distribution.
\begin{center}
\begin{table}[H]
 \ \ \ \ \ \ \  \ \ \ \ \ \ \includegraphics[height=40mm]{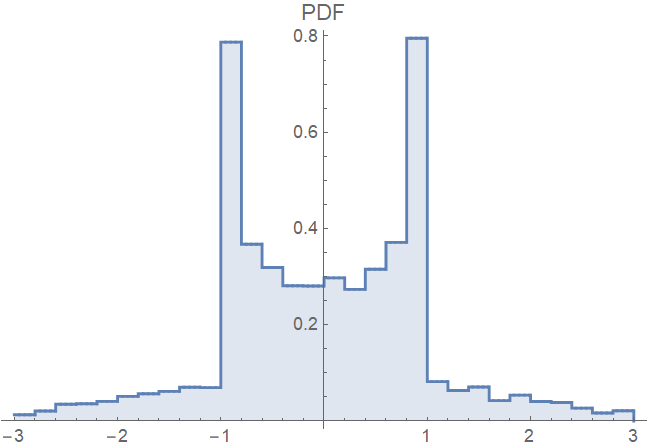}\ \   \ \ \ \ \ \ \ \  \ \ \ \ \ \ \ \ \includegraphics[height=45mm]{Histogram3F2Theoretical.png}\\
\caption*{ \ \ \ \ \ \ \ $A_\lambda$ histogram for $p=93283$ \ \ \ \ \ \ \ \ \ \ \ \ \ \ \ \ \ \ \ \ \ \ \  \ \ \ \ \ \ \ \ \ \ \ \ \ \ \ \ \ Plot of $\frac{1}{4\pi} f(t)$ \ \ \ \ \ \ \ \ \ \ }
\end{table}
\end{center}
\noindent
In view of these plots, we playfully refer to this distribution as the {\it Batman }distribution, and to the asymptotes as {\it Batman's ears}.

\end{example}

In analogy with the effective version of Birch's theorem by Murty and Prabhu stated in (\ref{EqMurty}), here we obtain an explicit version of this result.

\begin{theorem}\label{MainTheorem}
If $-3\leq a<b\leq 3$ and $p\geq 5$ is a prime, then
$$
\bigg|\frac{\#\left\{\lambda\in\F_{p} :A_\lambda(p)\in [a,b]\right\}}{p}-\frac{1}{4\pi}\int_a^b f(t)dt\bigg|\leq\frac{110.84}{p^{1/4}}.
$$
\end{theorem}

\begin{multRem}~
	
	\noindent
	(1) If $0<a<b<3$ or $-3<a<b<0,$ then the constant $110.84$ can be improved to $55.42.$
	
	\noindent
	(2) It is not difficult to generalize Theorem~\ref{MainTheorem} to arbitrary finite fields $\F_{p^r}.$ To do this, we note that the sums in Proposition~\ref{ClassNumberProp} should be modified for $\lambda$ such that $E_{\lambda}^{\Cl}$ is supersingular. The contribution of these supersingular curves can be bounded using elementary bounds for Hurwitz class numbers. The other results in this paper can then be applied almost verbatim. We leave these details to the interested reader.
\end{multRem}

We are motivated to obtain this explicit theorem due to the vertical asymptotes at $t=\pm1.$ For example,  if $T>0,$ then how large must the prime $p$ be so that the density of $A_{\lambda}(p)$ near $t=\pm 1$ is larger than $T?$ As the example above illustrates, $p=93283$ is not large enough for $T=1.$ More generally, numerical data suggests that the size of such $p$ must grow very rapidly with $T.$ Despite this growth, we have the following explicit answer.

\begin{corollary}\label{MainCorollary}
If $T>\frac{\sqrt{3}}{4\pi}, \delta>0,$ and $x(T,\delta)=\frac{4}{1+16\pi^2(T+\delta)^2},$  then the following are true.

\noindent
(1) If $p\geq\left(\frac{55.42}{x(T,\delta)\delta}\right)^4,$ then
$$
\frac{1}{x(T,\delta)}\cdot\frac{\#\left\{\lambda\in\F_{p} : A_\lambda(p)\in [1-x(T,\delta),1]\right\}}{p}>T.
$$

\noindent
(2) If $p\geq\left(\frac{55.42}{x(T,\delta)\delta}\right)^4,$ then
$$
\frac{1}{x(T,\delta)}\cdot\frac{\#\left\{\lambda\in\F_{p} : A_\lambda(p)\in [-1,-1+x(T,\delta)]\right\}}{p}>T.
$$

\noindent
Furthermore, the lower bound on $p$ is minimal when
$
\delta=\frac{\sqrt{16\pi^2T^2+1}}{4\pi}.
$
\end{corollary}

\begin{example}
Suppose that $T=10$ and $\delta=\frac{\sqrt{1600\pi^2+1}}{4\pi}.$ Corollary~\ref{MainCorollary} gives that if $p\geq 3.45\times 10^{14}$ is a prime and $x=0.00006\ldots,$ then we have
$$
\frac{1}{x}\cdot\frac{\#\left\{\lambda\in\F_{p} : A_\lambda(p)\in [1-x,1]\right\}}{p}>10.
$$
\end{example}

In order to prove Theorem~\ref{MainTheorem}, we use an explicit version of the method of moments. This requires a different set of tools than what was employed by Murty and Prabhu in \cite{murty}. Namely, Murty and Prabhu bound the errors in the moments of the Frobenius traces by applying bounds from the theory of holomorphic modular forms to an Eichler-Selberg trace formula. In contrast to \cite{murty}, the moments of the values $A_{\lambda}(p)$ arise in a very different manner, which necessitates most of the work in this paper. Namely, these moments arise in the theory of harmonic Maass forms, where there are no Eichler-Selberg type trace formulas per se. To make this precise, we recall the strategy employed to prove the ineffective version of Theorem~\ref{MainTheorem}.

The values $A_{\lambda}(p)$ are first expressed in terms of Frobenius traces $a_{\lambda}^{\Cl}(p)$ of the Clausen family of elliptic curves (see (\ref{ClausenFamily})). The elliptic curves are then grouped into isomorphism classes whose order is determined by the work of Schoof \cite{schoof}. This expresses the moments of the traces as weighted sums of Hurwitz class numbers. To estimate these moments, the weighted sums are expressed as Fourier coefficients of Rankin-Cohen brackets applied to Zagier's weight $\frac{3}{2}$ Eisenstein series and certain univariate theta functions. This work gives the asymptotic properties of the moments.

In this paper, we make these more precise, by explicitly bounding the error terms. To do this, we apply Deligne's Theorem, that bounds the Fourier coefficients of newforms, to the holomorphic projections of the Rankin-Cohen brackets above. This process is rather lengthy and involved, and makes use of the Rankin-Selberg unfolding argument to explicitly write the holomorphic projections of the Rankin-Cohen brackets in terms of newforms. 

This paper is organized as follows. In Section~\ref{SectionPreviousResults}, we recall important facts about the family $\{X_\lambda\}$ of $K3$ surfaces and the arithmetic of Clausen elliptic curves. In Section~\ref{SectionBackground}, we review necessary background on the theory of modular forms. In Section~\ref{SectionPetersson}, we find the Petersson inner products of newforms with their images under certain $V(d)$ operators. We then use the results from Sections~\ref{SectionBackground} and ~\ref{SectionPetersson} in Section~\ref{SectionCoefficients} to explicitly express the holomorphic projection of the Rankin-Cohen brackets alluded to above in terms of newforms. This gives us bounds on certain Fourier coefficients of these brackets. These bounds are used in Section~\ref{SectionProbTheory} to conclude certain explicit circular distributions from explicit moments, which are then connected to the $O3$ distribution. Finally, in Section~\ref{SectionProofMainResults}, we conclude with the proofs of Theorem~\ref{MainTheorem} and Corollary~\ref{MainCorollary}.

\section*{Acknowledgements}
\noindent
The author thanks Ken Ono for his guidance, his many helpful remarks and for providing research support with the Thomas Jefferson Fund and the NSF Grant (DMS-2002265 and DMS-2055118). The author would also like to thank Eleanor McSpirit for her comments which improved the quality of the exposition in this paper.

\section{Elliptic curves and the arithmetic of $\{X_\lambda\}$}\label{SectionPreviousResults}

To study the distribution of the values $A_{\lambda}(p)$ in Theorem~\ref{MainTheorem}, we interpret these values in terms of the Clausen elliptic curves
\begin{equation}\label{ClausenFamily}
E^{\Cl}_\lambda:\ \ \ y^2=(x-1)(x^2+\lambda).
\end{equation}
This interpretation is given by the following theorem. 

\begin{theorem}\label{K3ClausenCon}[Prop. 4.1 and Th. 2.1 of \cite{onoK3}]
If $p\geq 5$ is a prime, $\lambda\in\F_p\setminus\{0,-1\},$ and $a^{\Cl}_\lambda(p):=p+1-|E_\lambda^{\Cl}(\F_p)|,$ then we have
$$
p+p\cdot\phi(-\lambda)A_{-1-\frac{1}{\lambda}}(p)=a^{\Cl}_{\lambda}(p)^2,
$$
where $\phi$ is the quadratic character for $\F_p.$
\end{theorem}

To obtain our theorem, we need asymptotics with explicit error terms for the moments of $A_{\lambda}(p).$ Theorem~\ref{K3ClausenCon}  shows that these are essentially even power moments and twisted even power moments of the traces $a_{\lambda}(p).$ 
Using the work of Schoof \cite{schoof}, those power moments can be expressed as weighted sums of Hurwitz class numbers. 

To make this precise, we first fix some notation. If $-D<0$ such that $-D\equiv0,1\pmod{4},$ then denote by $\mathcal{O}(-D)$ the unique imaginary quadratic order with discriminant $-D.$ Let $h(D)$ denote the order of the class group of $\mathcal{O}(-D)$ and let $\omega(D)$ denote half the number of roots of unity in $\mathcal{O}(-D).$ In this notation, define\footnote{We define  $H^{\ast}(0):=-\frac{1}{12}$ and $h(D)=H^{\ast}(D)=0$ whenever $-D$ is neither zero nor a negative discriminant.}
$$
H^{\ast}(D):=\sum\limits_{f^2\mid D}\frac{h(D/f^2)}{\omega(D/f^2)}.
$$
Furthermore, to determine the power moments of the traces $a_{\lambda}^{\Cl},$ where $j(E_\lambda^{\Cl})=1728,$ we consider the decomposition of a prime $p$ as a sum of two squares. More precisely, if $p$ is a prime and $p\equiv 1\pmod 4,$ then we write $p=a^2+b^2$ for some unique choice of positive integers $a$ and $b$ with $a$ chosen to be odd. In this notation, if $n$ is a positive integer, then define 
\begin{equation}
c^{\pm}(p,n):=\begin{cases}
\frac{1}{2}((2a)^{2n}\pm (2b)^{2n}) &\text{\it{if} } p\equiv 1\pmod 4, \\
0 & \text{otherwise}.
\end{cases}
\end{equation}
The power moments we require are given by the following proposition.

\begin{proposition}\label{ClassNumberProp}
If $p\geq 5$ is a prime, then the following are true.

\noindent
(1) If $n$ is a positive integer, then
\begin{equation}\label{Eq1CN}
\sum\limits_{\lambda\in\F_p\setminus\{0,-1\}} a_\lambda^{\Cl}(p)^{2n}=\sum\limits_{\substack{0<s<2\sqrt{p} \\ 2\mid s}} \left(2H^\star\left(\frac{4p-s^2}{4}\right)+H^\star\left(4p-s^2\right)\right)s^{2n}-c^{+}(p,n).
\end{equation}

\noindent
(2) If $n$ is a positive integer, then
\begin{equation}\label{Eq2CN}
\sum\limits_{\lambda\in\F_p\setminus\{0,-1\}} \phi(-\lambda)a_\lambda^{\Cl}(p)^{2n}=\sum\limits_{\substack{0<s<2\sqrt{p} \\ 2\mid s}}\left(4H^\star\left(\frac{4p-s^2}{4}\right)-H^\star\left(4p-s^2\right)\right)s^{2n}-c^{-}(p,n).
\end{equation}
\end{proposition}

\begin{proof}

The sums over $\lambda$ such that $a_{\lambda}^{\Cl}(p)\neq a_{1/3}^{\Cl}(p),a_{-1/9}^{\Cl}(p)$ are given by Proposition 3.2 of \cite{BatmanPaper}. The remaining terms are computed in an analogous fashion using equations (7.21) and (7.25) of \cite{OnoFrechettePapanikolas}.
\end{proof}

\section{Background on Modular Forms}\label{SectionBackground}

To estimate the even power moments and twisted even power moments of the values $a_{\lambda}^{\Cl}(p),$ we study the weighted sums of Hurwitz class numbers that appear in Proposition~\ref{ClassNumberProp}. These weighted sums arise in the expressions of Fourier coefficients of harmonic Maass forms. To estimate these Fourier coefficients, we make use of the method of holomorphic projection to produce holomorphic cusp forms with similar coefficients. To obtain an explicit bound, we decompose these cusp forms in terms of newforms using the Petersson inner product and unfolding arguments due to Rankin and Selberg. To this end, this section recalls necessary tools from the theory of harmonic Maass forms, holomorphic modular forms, and Rankin-Selberg theory.

\subsection{Harmonic Maass Forms and Holomorphic Projection}\label{HMF&HolProj}

To exploit the connection between the weighted sums in (\ref{Eq1CN}) and (\ref{Eq2CN}) and the Fourier coefficients of harmonic Maass forms, we make use of the generating function $\mathcal{H}^+(\tau)$ of the Hurwitz class numbers. Namely, a construction of Zagier shows that $\mathcal{H}^+(\tau)$ is the {\it holomorphic part} of a harmonic Maass form $\mathcal{H}(\tau)$ of weight $\frac{3}{2}.$ The sums in equations~(\ref{Eq1CN}) and (\ref{Eq2CN}) then appear in the expressions of the Fourier coefficients of Rankin-Cohen brackets of $\mathcal{H}(\tau)$ and univariate theta functions. In this case, the holomorphic projections of these brackets are explicitly known due to work of Mertens \cite{mertensPhD}. Here we recall Zagier's construction, the definition and properties of Rankin-Cohen brackets, and Mertens' explicit expression of the holomorphic projections of these brackets.

To carry out Zagier's construction, we first recall notation from the theory of modular forms of half-integral weight. If $\gamma=\begin{pmatrix}
a & b \\
c & d
\end{pmatrix}\in\Gamma_0(4)$ and $\tau\in\HH:=\left\{x+iy\in\C, y>0\right\},$ then the automorphy factor is defined by
$$
j(\gamma,\tau):=\frac{\theta(\gamma\tau)}{\theta(\tau)},
$$
where $\gamma\tau:=\frac{a\tau+b}{c\tau+d}, q_{\tau}:=e^{2\pi i\tau},$ and
\begin{equation}
\theta(\tau):=\sum\limits_{n\in\Z} q_{\tau}^{n^2}.
\end{equation}
We recall the metaplectic extension of $\GL_2^+(\Q)$ defined as
$$
\Mp_2(\Q):=\left\{(\gamma,\phi): \gamma=\begin{pmatrix}
a & b \\ 
c & d
\end{pmatrix}\in\GL_2^+(\Q),
 \phi:\HH\to\C\text{ holomorhic, }\phi(\tau)^2=\frac{c\tau+d}{\sqrt{ad-bc}}\right\}.
$$
$\Mp_2(\Q)$ acts on $\{f:\HH\to\C\}$ through the slash operator. Namely, if $f:\HH\to\C, k$ is an integer, and $\tilde{\gamma}=\left(\begin{pmatrix}
a & b \\
c & d
\end{pmatrix},\phi\right)\in\Mp_2(\Q),$ then the (weight $\frac{k}{2}$) slash operator is defined\footnote{The dependence on $k$ is usually dropped from the notation if understood from context.} by
$$
\left(f|_{\frac{k}{2}}\tilde{\gamma}\right)(\tau)=\phi(\tau)^{-k}f\left(\frac{a\tau+b}{c\tau+d}\right).
$$
Finally, if $N$ is a positive integer with $4|N$, we define the Fricke involution 
$$
W_N:=\left(\begin{pmatrix}
0 & -1 \\
N & 0
\end{pmatrix},\sqrt{N^{1/2}\tau}\right),
$$ where $\sqrt{\cdot}$ is the principal branch of the square root. 

In this notation,  if $s\in\C$ with $\Re(s)>\frac{1}{4}$ and $\tau\in\HH,$ then define the Eisenstein series $E_{\frac{3}{2},s}(\tau)$ by
$$
E_{\frac{3}{2},s}(\tau):=\sum\limits_{\gamma\in\Gamma_\infty\backslash\Gamma_0(4)}\frac{1}{j(\gamma,\tau)^3}\cdot\Im(\gamma\tau)^s,
$$
where $\Gamma_\infty\leq \SL_2(\Z)$  is the stabilizer group of $i\infty$ and
define the Eisenstein series $F_{\frac{3}{2},s}(\tau)$ by
$$
F_{\frac{3}{2},s}(\tau):=(E_{\frac{3}{2},s}|_{\frac{3}{2}}W_4)(\tau).
$$
$E_{\frac{3}{2},s}(\tau)$ and $F_{\frac{3}{2},s}(\tau)$ are analytic maps in $s$ and have analytic continuations to $s=0,$ which we denote by $E_{\frac{3}{2}}(\tau)$ and $F_{\frac{3}{2}}(\tau)$ respectively. In this notation, Zagier proves the following theorem.

\begin{theorem}\label{ZagierHMSTheorem}\cite{Zagier}
The function
$$
\mathcal{H}(\tau):=\sum\limits_{n=0}^\infty H^\ast(n)q_\tau^n+\frac{1}{8\pi\sqrt{y}}+\frac{1}{4\sqrt{\pi}}\sum\limits_{n=1}^\infty n\Gamma(-\frac{1}{2}; 4\pi n^2y)q_{\tau}^{-n^2},
$$
where $\tau=x+iy\in \HH$ and $q_\tau:=e^{2\pi i\tau},$ is a weight $\frac{3}{2}$ harmonic Maass form with manageable growth at the cusps of $\Gamma_0(4).$ 
In fact, we have\footnote{This explicit form is computed in \cite{Zagier2}. A different normalization of $E_{\frac{3}{2}}(\tau)$ and $F_{\frac{3}{2}}(\tau)$ is used.}
$$
\mathcal{H}(\tau)=-\frac{1}{12}\left(E_{\frac{3}{2}}(\tau)+(1-i)2^{-3/2}F_{\frac{3}{2}}(\tau)\right).
$$
\end{theorem}

To obtain the sums in equations~(\ref{Eq1CN}) and (\ref{Eq2CN}) from $\mathcal{H}(\tau),$ we take Rankin-Cohen brackets of $\mathcal{H}(\tau)$ and univariate theta functions. To make this precise, let $f$ and $g$ be smooth functions defined on $\HH$ and let $k,l\in\frac{1}{2}\N$ and $m\in\N_0.$ The $m$th Rankin-Cohen bracket (of weight $(k,l)$) of $f$ and $g$ is
$$
[f,g]_m:=\frac{1}{(2\pi i)^m}\sum\limits_{r+s=m}(-1)^r\binom{k+m-1}{s}\binom{l+m-1}{r}\frac{d^r}{d\tau^r}f\cdot\frac{d^s}{d\tau^s}g.
$$ 
As the next proposition illustrates, these operators preserve modularity.

\begin{proposition}[Th. 7.1 of \cite{cohen}]\label{BracketProposition}
Let $f$ and $g$ be (not necessarily holomorphic) modular forms of weights $k$ and $l,$ respectively on a subgroup $\Gamma$ of $\Mp_2(\Q).$ Then the following are true.

\noindent
(1) We have that $[f,g]_m$ is modular of weight $k+l+2m$ on $\Gamma.$ 

\noindent
(2) If $\tilde{\gamma}\in \Mp_2(\Q),$ then
under the usual modular slash operator we have
$$
[f|_k\tilde{\gamma},g|_l\tilde{\gamma}]_m=([f,g]_m)|_{k+l+2m}\tilde{\gamma}.
$$
\end{proposition}

In this notation, the weighted class number sums appear in Fourier coefficients of $[\mathcal{H}(\tau),\theta(\tau)]_m$ and $[\mathcal{H}(\tau),\theta(4\tau)]_m.$ More precisely, if $\mathcal{H}^+(\tau):=\sum\limits_{n=0}^\infty H^\ast(n)q_{\tau}^n$ denotes the generating function of the Hurwitz class numbers, then we have the following lemma.

\begin{lemma}\label{ThetaBracketsLemma}
	
The following are true.
	
	\noindent
	(1) Let $
	[\mathcal{H}^+(\tau),\theta(\tau)]_m=\sum\limits_{n\geq 0}a_m(n)q_{\tau}^n.
	$
	Then, there exists constants $A(l,m,p)$ such that
	$$
	a_m(p)=\sum\limits_{l=0}^m A(l,m,p)\sum\limits_{\substack{0<s<2\sqrt{p} \\ 2\mid s}}H^\star\left(\frac{4p-s^2}{4}\right)s^{2l}.
	$$

	\noindent
	(2) Let  $[\mathcal{H}^+(\tau),\theta(4\tau)]_m=\sum\limits_{n\geq 0}b_m(n)q_{\tau}^n.$ 
	Then, there exists constants $B(l,m,p)$ such that
	$$
	b_m(4p)=\sum\limits_{l=0}^m B(l,m,p)\sum\limits_{\substack{0<s<2\sqrt{p} \\ 2\mid s}} H^\star(4p-s^2)s^{2l}.
	$$
\end{lemma}

	This follows directly from the definition of the Rankin-Cohen bracket. The constants $A(l,m,p)$ and $B(l,m,p)$ are  expressed explicitly in Proposition~\ref{ChebyshevProp}.

To estimate the Fourier coefficients of $[\mathcal{H}(\tau),\theta(\tau)]_m$ and $[\mathcal{H}(\tau),\theta(4\tau)]_m$, we produce a holomorphic modular form with similar Fourier coefficients by the method of holomorphic projection. To make this precise, suppose $h:\HH\to\C$ is a (not necessarily holomorphic) modular form of weight $k>2$ on $\Gamma_0(N)$ with Fourier expansion
$$
h(\tau)=\sum\limits_{n\in\Z} c_h(n,y)q_{\tau}^n,
$$
where $\tau=x+iy.$ Further, suppose that $h(\tau)$ has moderate growth at cusps, with $c_h(0,y)=c_0+O(y^{-\varepsilon})$ for some $\varepsilon>0.$ Then, the holomorphic projection of $h(\tau)$ is defined by
$$
(\pi_{\hol} h)(\tau):=c_0+\sum\limits_{n\geq 1} c(n)q_{\tau}^n,
$$
where
$$
c(n)=\frac{(4\pi n)^{k-1}}{(k-2)!}\int_{0}^{\infty}c_h(n,y)e^{-4\pi ny}y^{k-2}dy,
$$
for $n\geq 1.$  The following proposition explains the importance of the holomorphic projection operator.

\begin{proposition}\label{HolProjProp}[see Section 10.1 of \cite{HMFOnoBook}]
Assuming the hypothesis above, the following are true.

\noindent
(1) $\pi_{\hol}(f)\in\mathcal{M}_k(\Gamma_0(N)).$ 

\noindent
(2) If $g\in\mathcal{S}_k(\Gamma_0(N)),$ then we have
$
\langle f,g\rangle=\langle \pi_{\hol}(f),g\rangle,
$
whenever the left-hand side converges.
\end{proposition}

The following theorem of Mertens explicitly computes the holomorphic projection of the Rankin-Cohen brackets of $\mathcal{H}(\tau)$ with a univariate theta function of weight $\frac{1}{2}.$

\begin{theorem}\label{HolProjTh}[Th. V.2.1 of \cite{mertensPhD}]
If $m$ and $s$ are positive integers, then we have
$$
\pi_{\hol}\left([\mathcal{H}(\tau),\theta(s\tau)]_m\right)=[\mathcal{H}^+(\tau),\theta(s\tau)]_m+\frac{1}{2}\cdot\frac{\binom{2m}{m}}{4^m}\Lambda_s(\tau;m),
$$
where
$$
\Lambda_s(\tau;m)=2\sum\limits_{n\geq 1}\left(\sum\limits_{\substack{st^2-r^2=n \\ t,r\geq 1}}(\sqrt{s}t-r)^{2m+1}\right)q_{\tau}^n+\sum\limits_{n\geq 1}(\sqrt{s}n)^{2m+1}q_{\tau}^{sn^2},
$$
and $\mathcal{H}^+(\tau):=\sum\limits_{n=0}^\infty H^\ast(n)q_{\tau}^n.$
\end{theorem}

\subsection{Holomorphic Modular Forms}\label{SubsModForms}

To obtain an explicit estimate of the Fourier coefficients of   $\pi_{\hol}([\mathcal{H}(\tau),\theta(\tau)]_m)$ and $\pi_{\hol}([\mathcal{H}(\tau),\theta(4\tau)]_m),$ we decompose these cusp forms in terms of newforms. Here we give definitions and recall fundamental facts about holomorphic modular forms of integral weight.

If $k$ and $N$ are positive integers, then denote by $\mathcal{M}_k(\Gamma_0(N))$ the vector space of holomorphic modular forms of weight $k$ and level $N.$ The cuspidal subspace of $\mathcal{M}_k(\Gamma_0(N))$ is denoted by $\mathcal{S}_k(\Gamma_0(N)).$

$\GL_2^+(\Q)$ acts on $\{f:\HH\to\C\}$ through the slash operator. Namely, if $f:\HH\to\C, k$ is an integer, and $\gamma=\begin{pmatrix}
a & b \\
c & d
\end{pmatrix}\in\GL_2^+(\Q),$ then the (weight $k$) slash operator is defined by
$$
(f|_k\gamma)(\tau)=(ad-bc)^{k/2}(c\tau+d)^{-k}f\left(\frac{a\tau+b}{c\tau+d}\right).
$$
In our setting, two matrices in $\GL_2^+(\Q)$ play a significant role, namely, the matrix
$
V(d):=\begin{pmatrix}
d & 0 \\
0 & 1
\end{pmatrix}
$
and the Fricke involution
$
W_N:=\begin{pmatrix}
0 & -1 \\
N & 0
\end{pmatrix}.
$ The actions of these matrices send cusp forms to cusp forms of possibly higher level. More precisely, if $d$ is a positive integer and $f\in\mathcal{S}_k(\Gamma_0(N)),$ then $f|V(d)\in\mathcal{S}_k(\Gamma_0(dN))$ and $f|W_N\in\mathcal{S}_k(\Gamma_0(N)).$ Furthermore, if $f$ is a cusp form with Fourier series  
$
f(\tau)=\sum\limits_{n\geq 1} a_f(n)q_{\tau}^n,
$
where $q_{\tau}=e^{2\pi i\tau},$ then we write
$$
(f|U(d))(\tau):=\sum\limits_{n\geq 1} a_f(nd)q_{\tau}^n.
$$

The vector space of cusp forms can be given the structure of a finite-dimensional Hilbert space using the Petersson inner product. To make this precise, if $f,g\in\mathcal{S}_k(\Gamma)$ for some subgroup $\Gamma$ of $\SL_2(\Z)$ of finite index, then define\footnote{We drop the $\Gamma$ when it is understood from context.} the Petersson inner product\footnote{There are different normalizations of the Petersson inner product.} of $f$ and $g$ by
$$
\langle f,g\rangle_{\Gamma}:=\frac{1}{[\SL_2(\Z):\Gamma]}\iint_{\HH/\Gamma} f(x+iy)\overline{g(x+iy)}y^k\frac{dxdy}{y^2}.
$$
The slash operator is an isometry of Hilbert spaces. In other words, if $\Gamma\subset\SL_2(\Z)$ is a subgroup of finite index and $\gamma\in\GL_2^+(\Q),$ then
$$
\langle f,g\rangle_\Gamma=\langle f|\gamma,g|\gamma\rangle_{\Gamma'},
$$
where $\Gamma'=\gamma^{-1}\Gamma \gamma\cap\Gamma.$ In particular, this proves that if $N$ is a positive integer and $f,g\in\mathcal{S}_k(\Gamma_0(N)),$ then 
$$
\langle f,g\rangle=\langle f|W_N,g|W_N\rangle.
$$

As seen above, modular forms of a level $N$ can be obtained by applying $V(d)$ operators to modular forms of lower level. This decomposes the space $\mathcal{S}_k(\Gamma_0(N))$ into an old subspace and a new subspace. To make this precise, we define the space of oldforms
$$
\mathcal{S}_k^{\text{old}}(\Gamma_0(N)):=\sum\limits_{\substack{M<N \\ d|N/M}} \mathcal{S}_k(\Gamma_0(M))|V(d).
$$
Naturally, we define the new subspace $\mathcal{S}_k^{\text{new}}(\Gamma_0(N))$ as the orthogonal complement of $\mathcal{S}_k^{\text{old}}(\Gamma_0(N))$ in $\mathcal{S}_k(\Gamma_0(N)).$ An important property of the space $\mathcal{S}_k^{\text{new}}(\Gamma_0(N))$ is that it has a basis of eigenforms of {\it all} Hecke operators. If $f(\tau)=\sum\limits_{n\geq 1}a_f(n)q_{\tau}^n$ is such an eigenform with $a_f(1)=1,$ we call $f$ a normalized newform. 

The decomposition of $\mathcal{S}_k(\Gamma_0(N))$ in terms of normalized newforms is crucial to prove Theorem~\ref{MainTheorem}. Namely, a celebrated theorem of Deligne explicitly bounds the Fourier coefficients of normalized newforms.

\begin{theorem}\label{DeligneTheorem}\cite{deligne}
If $f(\tau)=\sum\limits_{n\geq 1}a_f(n)q_{\tau}^n\in\mathcal{S}_k^{\text{new}}(\Gamma_0(N))$ is a normalized newform, then $|a_f(n)|\leq d(n)n^{\frac{k-1}{2}},$ where $d(n)$ is the divisor function.
\end{theorem}

Furthermore, to compute inner products of newforms and given cusp forms, we make use of the following property of newforms.

\begin{theorem}[Th. 13.3.9 of \cite{CohenStromberg}]\label{Mult1}
Let $f\in\mathcal{S}_k^{\text{new}}(\Gamma_0(N))$ be a newform and $g\in\mathcal{S}_k(\Gamma_0(M)).$ If $g$ is an eigenfunction of the Hecke operators $T(p)$ for all $p\nmid N$ with the same eigenvalues as $f$ for all but finitely many $p,$ then $g=\lambda f$ for some $\lambda\in\C.$
\end{theorem}

Theorem~\ref{Mult1} is generally referred to as the multiplicity-one theorem.  It implies the following corollary which we implicitly use throughout Sections~\ref{SectionPetersson} and \ref{SectionCoefficients}.

\begin{corollary}\label{Mult1Corollary}
If $f\in\mathcal{S}_k^{\text{new}}(\Gamma_0(N))$ is a normalized newform, then the following are true.

\noindent
(1) There exists $\lambda_N(f)\in\{\pm1\}$ such that
$
f|W_N=\lambda_N(f)f.
$

\noindent
(2) If $g\in\mathcal{S}_k^{\text{new}}(\Gamma_0(N))$ is a normalized newform distinct from $f,$ then $\langle f,g\rangle=0.$
\end{corollary}

\begin{proof}

To show (1), apply part (b) of Proposition 13.3.11 of \cite{CohenStromberg} with $Q=N.$

To show (2), note that since $g$ is not a multiple of $f,$ Theorem~\ref{Mult1} implies that there exists some prime $p\nmid N$ such that $T(p)f\neq T(p)g.$ Since the Hecke operators are self-adjoint (see Theorem 10.3.5 of \cite{CohenStromberg}) with respect to the Petersson inner product on $\mathcal{S}_k(\Gamma_0(N)),$ this shows that $\langle f,g\rangle=0.$
\end{proof}

\subsection{Rankin-Selberg Theory}

To decompose $\pi_{\hol}([\mathcal{H}(\tau),\theta(\tau)]_m)$ and $\pi_{\hol}([\mathcal{H}(\tau),\theta(4\tau)]_m)$ as explained above, we take the inner products of these cusp forms with newforms. To compute these inner products, we make use of the Hecke relations between Fourier coefficients of a given newform. To this end, we use unfolding theorems due to Rankin and Selberg (see section 11.12 of \cite{CohenStromberg}). The following theorem relates the Petersson inner product of two cusp forms to their respective Fourier coefficients.

\begin{theorem}\label{Rankin1}[see \cite{RankinContributions}]
If $f(\tau)=\sum\limits_{n\geq 1}a_f(n)q_{\tau}^n$ and $g(\tau)=\sum\limits_{n\geq 1}a_g(n)q_{\tau}^n$ are cusp forms of weight $k$ on a congruence subgroup $\Gamma,$ then we have
$$
\langle f,g\rangle=\frac{\pi}{3}\frac{(k-1)!}{(4\pi)^k}\lim\limits_{x\to\infty}\frac{1}{x}\sum\limits_{n\leq x}\frac{a_f(n)\overline{a_g(n)}}{n^{k-1}}.
$$
\end{theorem}

In our setting, we study inner product of cusp forms with the brackets $[\mathcal{H}(\tau),\theta(\tau)]_m$ and $[\mathcal{H}(\tau),\theta(4\tau)]_m.$ Since $\mathcal{H}(\tau)$ is defined in terms of Eisenstein series, the following proposition gives the unfolding result we require.

\begin{proposition}\label{UnfoldingInGeneral}
If $m$ is a positive integer, $f(\tau)=\sum\limits_{n\geq 1}a_f(n)q_{\tau}^n\in\mathcal{S}_{2m+2}(\Gamma_0(4)),$ and $g(\tau)=\sum\limits_{n\geq 1}a_g(n)q_{\tau}^n\in\mathcal{S}_{1/2}(\Gamma_0(4)),$ then we have
$$
\langle f , [E_{\frac{3}{2}},g]_m\rangle=\frac{1}{6}\binom{m-\frac{1}{2}} {m}\frac{(2m)!}{(4\pi)^{2m+1}}\sum\limits_{n\geq 1} \frac{a_f(n)\overline{a_g(n)}}{n^{m+1}}.
$$
\end{proposition}

\begin{proof}[Sketch of Proof]

Ignoring convergence issues, we have that
\begin{align*}
[E_{\frac{3}{2},s}(\tau),g(\tau)]_m&=\sum\limits_{\gamma\in\Gamma_\infty\backslash\Gamma_0(4)}[j(\gamma,\tau)^{-3}\cdot\Im(\gamma\tau)^s,g(\tau)]_m\\
&=\sum\limits_{\gamma\in\Gamma_\infty\backslash\Gamma_0(4)}[\Im(\tau)^s|_{\frac{3}{2}}\gamma,\left((g|_{1/2}\gamma^{-1})|_{1/2}\gamma\right)(\tau)]_m.
\end{align*}
Since $g$ is modular of weight $\frac{1}{2}$ on $\Gamma_0(4),$ Proposition~\ref{BracketProposition} gives us that
$$
[E_{\frac{3}{2},s}(\tau),g(\tau)]_m=\sum\limits_{\gamma\in\Gamma_\infty\backslash\Gamma_0(4)}[\Im(\tau)^s,g(\tau)]_m|_{2m+2}\gamma.
$$
By the definition of the Rankin-Cohen bracket, a simple computation then gives that
$$
[E_{\frac{3}{2},s}(\tau),g(\tau)]_m=\sum\limits_{\gamma\in\Gamma_\infty\backslash\Gamma_0(4)}\sum\limits_{l=0}^{m}{\binom{m+\frac{1}{2}}{m-l}}{\binom{m-\frac{1}{2}}{l}}\frac{s!}{(s-l)!}\frac{\Im(\gamma\tau)^{s-l}}{(c\tau+d)^{2m+2}}\frac{1}{(2\pi i)^{m-l}}\frac{\partial^{m-l}g}{\partial\tau^{m-l}}(\gamma\tau).
$$
By the definition of the Petersson inner product and modularity of $f,$ we  have that
$$
\langle [E_{\frac{3}{2},s},g]_m,f\rangle=\sum\limits_{l=0}^m\frac{1}{(2\pi i)^{m-l}} {\binom{m+\frac{1}{2}}{m-l}}{\binom{m-\frac{1}{2}}{l}}\frac{s!}{(s-l)!}\iint_{\HH/\Gamma_\infty}\overline{f(\tau)}\cdot\frac{\partial^{m-l}g}{\partial\tau^{m-l}}(\tau)\cdot y^{2m+s-l}dxdy,
$$
where $\tau=x+iy.$ The rest of the proof follows by writing $f(\tau)$ and $g(\tau)$ as Fourier series, computing the resulting elementary integral, and taking analytic continuation to $s=0$ on both sides. 
\end{proof}

Additionally, we will need an analogue of Proposition~\ref{UnfoldingInGeneral} when $f$ and $g$ are of level dividing $16.$ Since $E_{\frac{3}{2}}(\tau)$ is the Eisenstein series of weight $\frac{3}{2}$ for $\Gamma_0(4),$ the above proposition doesn't hold. However, we have the following similar result.

\begin{proposition}\label{UnfoldingInGeneral2}
If $m$ is a positive integer, $f\in\mathcal{S}_{2m+2}(\Gamma_0(16)),$ and $g\in\mathcal{M}_{1/2}(\Gamma_0(16)),$ then we have
$$
\langle f,[E_{\frac{3}{2}},g]_m\rangle=\frac{1}{24}{\binom{m-\frac{1}{2}} {m}}\frac{(2m)!}{(4\pi)^{2m+1}}\sum\limits_{A\in S}\sum\limits_{n\geq 1}\frac{a_{f,A}(n)\overline{a_{g,A}(n)}}{n^{m+1}},
$$
where $S:=\left\{\begin{pmatrix}
1 & 0 \\
0 & 1
\end{pmatrix},\begin{pmatrix} 
1 & 0 \\
4 & 1
\end{pmatrix}, \begin{pmatrix}
3 & -1 \\
4 & -1 \\
\end{pmatrix}, \begin{pmatrix}
1 & 0 \\ 
8 & 1
\end{pmatrix}
\right\}, (f|A)(\tau)=\sum\limits_{n\geq 1}a_{f,A}(n)q_{\tau}^n,$ and $(g|A)(\tau)=\sum\limits_{n\geq 1}a_{g,A}(n)q_{\tau}^n.$
\end{proposition}

\begin{proof}
If $s\in\C$ with $\Re(s)>\frac{1}{4}$ and $\tau\in\HH,$ then define 
\begin{equation}
G_{\frac{3}{2},s}(\tau):=\sum\limits_{\gamma\in\Gamma_\infty\backslash\Gamma_0(16)}\frac{1}{j(\gamma,\tau)^3}\cdot\Im(\gamma\tau)^s,
\end{equation}
where $j(\gamma,\tau)$ is as in Subsection~\ref{HMF&HolProj}. It is then clear that
$$
E_{\frac{3}{2},s}(\tau)=\sum\limits_{A\in S} (G_{\frac{3}{2},s}|_{\frac{3}{2}}A^{-1})(\tau).
$$
The proof follows exactly as in Proposition~\ref{UnfoldingInGeneral}.
\end{proof}

\begin{remark}
If $G_{\frac{3}{2}}(\tau)$ denotes the analytic continuation of $G_{\frac{3}{2},s}(\tau)$ to $s=0,$ then Proposition~\ref{UnfoldingInGeneral} holds with $G_{\frac{3}{2}}$ and $\Gamma_0(16)$ in place of $E_{\frac{3}{2}}$ and $\Gamma_0(4).$
\end{remark}

To compute inner products of the form $\langle [F_{3/2},g]_m,f\rangle$ using Propositions~\ref{UnfoldingInGeneral} and ~\ref{UnfoldingInGeneral2}, we consider the action of the Fricke involution on Eisenstein series of half-integral weight and theta functions. The following lemma gives the computations we require.

\begin{lemma}\label{FrickeProps}
The following are true.

\noindent
(1) We have
$$
E_{\frac{3}{2}}|(W_4^2)=iE_{\frac{3}{2}};\ \ \ \ \ \ \theta|W_4=e^{-\frac{i\pi}{4}}\theta;\ \ \ \ \ \ \theta(4\tau)|W_{16}=\frac{1}{\sqrt{2}}e^{-\frac{i\pi}{4}}\theta(\tau).
$$

\noindent
(2) If $G_{\frac{3}{2}}$ denotes the analytic continuation of $G_{\frac{3}{2},s}$ at $s=0,$ then we have
$$
F_{\frac{3}{2}}=2^{3/2}\cdot G_{\frac{3}{2}}|W_{16}\ \ \ \text{ and }\ \ \ \ 
G_{\frac{3}{2}}|(W_{16}^2)=iG_{\frac{3}{2}}.
$$
\end{lemma} 

\begin{proof}
The proof of (1) is trivial.  On the other hand, note that $G_{\frac{3}{2},s}(\tau)=E_{\frac{3}{2},s}(4\tau),$ and therefore, by analytic continuation, $G_{\frac{3}{2}}(\tau)=E_{\frac{3}{2}}(4\tau).$ The proof of (2) reduces to an elementary computation.
\end{proof}

In the case where $f$ is a newform and $g$ is a univariate theta function, the following lemma expresses the series in Propositions~\ref{UnfoldingInGeneral} and ~\ref{UnfoldingInGeneral2} in terms of the Petersson norm of $f.$

\begin{lemma}\label{IntermediateUnfolding}
If $f(\tau)=\sum\limits_{n\geq 1}a_f(n)q_{\tau}^n$ is a normalized newform of weight $k$ and level $N,$  then we have
$$
\langle f,f\rangle=\frac{\pi}{3}\cdot\frac{(k-1)!}{(4\pi)^k}\cdot\prod\limits_{p\mid N}\left(1-\frac{1}{p}\right)\cdot\sum\limits_{n\geq 1}\frac{a_f(n^2)}{n^k}.
$$
\end{lemma}

\begin{proof}
It is well known (see Corollary 11.12.3 of \cite{cohen}) that
$
\sum\limits_{n\geq 1}\frac{|a_f(n)|^2}{n^s}
$
converges for $\Re(s)>k$ and has a simple pole at $s=k$ with residue
$
\frac{3}{\pi}\cdot\frac{(4\pi)^k}{(k-1)!}\langle f,f\rangle.
$
Since $f$ is a newform on $\Gamma_0(N)$ with trivial nebentypus, $a(n)$ is real for all $n$ and we have (see the introduction of \cite{ShimuraOnTheHolomorphy}) that
$$
\sum\limits_{n\geq 1}\frac{a_f(n)^2}{n^s}=\zeta_N(s-k+1)\sum\limits_{n\geq 1}\frac{a_f(n^2)}{n^s},
$$
where
$
\zeta_N(s)=\sum\limits_{\substack{n\geq 1 \\ (n,N)=1}}\frac{1}{n^s}.
$
The lemma follows by taking residues at $s=k.$
\end{proof}

\section{Petersson inner Product and the $V(d)$ Operator}\label{SectionPetersson}

To obtain the explicit bound in Theorem~\ref{MainTheorem}, we make use of upper bound on certain weighted sums of Hurwitz class numbers. In particular, we find an explicit bound on the Fourier coefficients given in Lemma~\ref{ThetaBracketsLemma}.
To obtain these explicit bounds, we require a partial decomposition of the brackets $\pi_{\hol}([\mathcal{H}(\tau),\theta(\tau)]_m)$ and $\pi_{\hol}([\mathcal{H}(\tau),\theta(4\tau)]_m)$ in terms of newforms. Namely, we decompose $\pi_{\hol}([\mathcal{H}(\tau),\theta(\tau)]_m)$ in the basis $\{f|V(d)\},$ where $f$ is a newform of level $N$ with $dN\mid 4.$ For $\pi_{\hol}([\mathcal{H}(\tau),\theta(4\tau)]_m),$ we require a similar partial decomposition. To obtain these decompositions, we take the inner product of these brackets with $f|V(d).$  Since $f$ and $f|V(d)$ are not always orthogonal when $d>1,$ we compute $\langle f,f|V(d)\rangle$ in terms of $\langle f,f\rangle.$ This depends on the level of $f.$ If $f$ is of level $4,$ we have the following lemma.

\begin{lemma}\label{VOperatorNewformLevel4}
If $f(\tau)=\sum\limits_{n\geq 1}a_f(n)q_{\tau}^n\in\mathcal{S}_k^{\text{new}}\left(\Gamma_0(4)\right)$ is a normalized newform and $l\geq 1,$ then 
$
\langle f,f|V(2^l)\rangle=0.
$ 
\end{lemma}

\begin{proof}

Theorem~\ref{Rankin1} implies that
\begin{align}
\langle f,f|V(2^l)\rangle\notag&=\frac{\pi}{3}\frac{(k-1)!}{(4\pi)^k}\lim\limits_{x\to\infty}\frac{1}{x}\sum\limits_{n\leq x}\frac{a_f(n)\cdot 2^{\frac{lk}{2}}\overline{a_f(\frac{n}{2^l})}}{n^{k-1}}\\
\notag&=\frac{\pi}{3}\frac{(k-1)!}{(4\pi)^k}\lim\limits_{x\to\infty}\frac{2^{-\frac{lk}{2}}}{x/2^l}\sum\limits_{n\leq x/2^l}\frac{a_f(2^ln)\overline{a_f(n)}}{n^{k-1}}.
\end{align}
It is well-known (see part (a) of Proposition 13.3.14 of \cite{CohenStromberg}) that the even coefficients of a normalized newform of level $4$ vanish, proving the claim.
\end{proof}

If $f$ is of level $2,$ we have the following lemma.

\begin{lemma}\label{VOperatorNewformLevel2}

If $f(\tau)=\sum\limits_{n\geq 1}a_f(n)q_{\tau}^n\in\mathcal{S}_k^{\text{new}}\left(\Gamma_0(2)\right)$ is a normalized newform and $l\geq 0,$ then
$$
\langle f,f|V(2^l)\rangle=\left(-\frac{\lambda_2\left(f\right)}{2}\right)^l\langle f,f\rangle.
$$
\end{lemma}

\begin{proof}

As in the proof of Lemma~\ref{VOperatorNewformLevel4}, we have that
$$
\langle f,f|V(2^l)\rangle=\frac{\pi}{3}\frac{(k-1)!}{(4\pi)^k}\cdot 2^{-\frac{lk}{2}}\lim\limits_{x\to\infty}\frac{1}{x}\sum\limits_{n\leq x}\frac{a_f(2^ln)\overline{a_f(n)}}{n^{k-1}}.
$$
Since $f$ is a normalized newform of level $2,$ we have that $a_f(2^{l}n)=a_f(2)^{l}\cdot a_f(n)$ for all $n\geq 1.$ This implies that
$$
\langle f,f|V(2^l)\rangle=\frac{a_f(2)^l}{2^{\frac{lk}{2}}}\cdot \frac{\pi}{3}\frac{(k-1)!}{(4\pi)^k}\lim_{x\to\infty}\cdot\frac{1}{x}\sum\limits_{n\leq x}\frac{a_f(n)\overline{a_f(n)}}{n^{k-1}}.
$$
Since $f$ is of prime level $2,$ the eigenvalue $\lambda_2(f)$ of the Fricke involution can be written in terms of the Fourier coefficient $a_f(2)$ of $f$ (see part (b) of Proposition 13.3.14 of \cite{CohenStromberg}). Namely, we have that $\lambda_2(f)=-2^{1-\frac{k}{2}}a_f(2).$  An application of Theorem~\ref{Rankin1} then concludes the proof.
\end{proof}

If $f$ is of level $1,$ we have the following lemma.

\begin{lemma}\label{VOperatorNewformLevel1}
If $f(\tau)=\sum\limits_{n\geq 1}a_f(n)q_{\tau}^n\in\mathcal{S}_k\left(\SL_2(\Z)\right)$ is a normalized newform, then the following are true.

\noindent
(1) 
$$
\langle f,f|V(2)\rangle=\frac{1}{3}\cdot\frac{a_f(2)}{2^{k/2-1}}\langle f,f\rangle.
$$

\noindent
(2) 
$$
\langle f,f|V(4)\rangle=\left(\frac{1}{6}\cdot\frac{a_f(2)^2}{4^{k/2-1}}-\frac{1}{2}\right)\langle f,f\rangle.
$$

\noindent
(3) 
$$
\langle f,f|V(8)\rangle=\left(\frac{1}{12}\cdot\frac{a_f(2)^3}{8^{k/2-1}}-\frac{5}{12}\cdot\frac{a_f(2)}{2^{k/2-1}}\right)\langle f,f\rangle.
$$

\noindent
(4) 
$$
\langle f,f|V(16)\rangle=\left(\frac{1}{24}\cdot\frac{a_f(2)^4}{16^{k/2-1}}-\frac{7}{24}\cdot\frac{a_f(2)^2}{4^{k/2-1}}+\frac{1}{4}\right)\langle f,f\rangle.
$$
\end{lemma}

\begin{proof}

The same argument in the proofs of Lemmas~\ref{VOperatorNewformLevel4} and ~\ref{VOperatorNewformLevel2} gives us that
$$
\langle f,f|V(2^l)\rangle= 2^{-\frac{lk}{2}}\langle f|U(2^l),f\rangle,
$$
for all $l\geq 0.$ Since $f$ is of level $1,$ the relation between $a(2^ln)$ and $a(n)$ is more involved. Therefore, we must determine the Hecke operator $T(2^l)$ in terms of $U$ and $V$ operators. We make this precise for (1), and leave the remaining cases to the reader.

Hecke's recurrence relations (see Corollary 10.4.4 of \cite{CohenStromberg}) imply that $T(2)f=f|U(2)+2^{\frac{k}{2}-1} f|V(2).$  On the other hand, since $f$ is a normalized eigenform of $T(2),$ we have that $T(2)f=a_f(2)f.$ Combining these two equations together, we have that
$$
a_f(2)f=f|U(2)+2^{\frac{k}{2}-1} f|V(2).
$$
Taking the inner product with $f$ on both sides, we find that
\begin{align}
a_f(2)\langle f,f\rangle\notag&=\langle f,f|U(2)\rangle+2^{\frac{k}{2}-1}\langle f,f|V(2)\rangle\\
\notag&=2^{\frac{k}{2}}\langle f,f|V(2)\rangle+2^{\frac{k}{2}-1}\langle f,f|V(2)\rangle.
\end{align}
This concludes the proof of (1). The proofs of the remaining statements are conceptually similar but more involved.
\end{proof}

\section{Explicit bounds for weighted sums of class numbers}\label{SectionCoefficients} 

To estimate the weighted class number sums in (\ref{Eq1CN}) and (\ref{Eq2CN}) , we observed that they appear in the expressions of Fourier coefficients of the cusp forms $\pi_{\hol}([\mathcal{H}(\tau),\theta(\tau)]_m)$ and $\pi_{\hol}([\mathcal{H}(\tau),\theta(4\tau)]).$ In this section we decompose these forms in the basis of newforms $f$ and their images under $V(d)$ operators and apply Theorem~\ref{DeligneTheorem} to bound the Fourier coefficients.

\subsection{Decomposition of $\pi_{\hol}([\mathcal{H}(\tau),\theta(\tau)]_{m})$}

Here we decompose $\pi_{\hol}([\mathcal{H}(\tau),\theta(\tau)]_m)$ for $m\in\N.$ To this end, we compute the inner product of $\pi_{\hol}([\mathcal{H}(\tau),\theta(\tau)]_m)$ with the elements of $\{f|V(d)\},$ where $f$ is a normalized newform of level $N$ and $dN|4.$  The first lemma concerns level $4.$

\begin{lemma}\label{ScalarThetaNewformLevel4}
If $f(\tau)=\sum\limits_{n\geq 1}a_f(n)q_{\tau}^n \in\mathcal{S}_{2m+2}^{\text{new}}\left(\Gamma_0(4)\right)$ is a normalized newform, then
$$
\langle \pi_{\hol}([\mathcal{H},\theta]_{m}),f\rangle=-\frac{1}{3}\cdot\frac{{\binom{2m}{m}}}{4^m}\langle f,f \rangle. 
$$
\end{lemma}

\begin{proof}
By Proposition~\ref{HolProjProp}, we have
$$
\langle \pi_{\hol}([\mathcal{H},\theta]_{m}),f\rangle=\langle [\mathcal{H},\theta]_{m},f \rangle. 
$$
We write $\mathcal{H}$ in terms of Eisenstein series to obtain
$$
\langle [\mathcal{H},\theta]_{m},f \rangle=-\frac{1}{12}\langle [E_{\frac{3}{2}},\theta]_{m},f \rangle-\frac{1}{12}\cdot\frac{1-i}{2^{3/2}}\langle [E_{\frac{3}{2}}|W_4,\theta]_{m},f\rangle.
$$
Since $W_{4}^2$ acts trivially on $\mathcal{M}_k(\Gamma_0(4))$ when $k$ is an even integer,  part (2) of Proposition~\ref{BracketProposition} shows that
$$
\langle [E_{\frac{3}{2}}|W_4,\theta]_{m},f\rangle=\langle [E_{\frac{3}{2}}|(W_4)^2,\theta|W_4]_{m},f|W_4\rangle.
$$
Using Lemma~\ref{FrickeProps}, we have that
$$
\langle [\mathcal{H},\theta]_m,f\rangle=-\frac{1}{12}\langle [E_{\frac{3}{2}},\theta]_m,f\rangle-\frac{1}{24}\langle [E_{\frac{3}{2}},\theta]_m, f|W_{4}\rangle.
$$
Since $f$ is a newform of level $4, f|W_4=-f$ (see Theorem 7 of \cite{AtkinLehner}). Therefore, we have
$$
\langle \pi_{\hol}([\mathcal{H},\theta]_{m}),f\rangle=-\frac{1}{24}\langle [E_{\frac{3}{2}},\theta]_m, f\rangle.
$$
Proposition~\ref{UnfoldingInGeneral} then gives that
$$
\langle \pi_{\hol}([\mathcal{H},\theta]_m),f\rangle=-\frac{1}{144}\cdot{\binom{m+\frac{1}{2}}{m} }\frac{(2m)!}{(4\pi)^{2m+1}}\sum\limits_{n\geq 1}\frac{2\cdot a_f(n^2)}{n^{2m+2}}.
$$
We apply Lemma~\ref{IntermediateUnfolding} to obtain
$$
\langle \pi_{\hol}([\mathcal{H},\theta]_m),f\rangle=-\frac{1}{72}\cdot{\binom{m+\frac{1}{2}}{m}}\frac{(2m)!}{(4\pi)^{2m+1}}\cdot 2\cdot\frac{(4\pi)^{2m+2}}{(2m+1)!}\cdot\frac{3}{\pi}\langle f,f\rangle.
$$
Since $\frac{{\binom{m+\frac{1}{2}}{ m}}}{2m+1}=\frac{{\binom{2m}{m}}}{4^m}, $ the claim follows.
\end{proof}

If $f$ is of level $2,$ we have the following lemma.

\begin{lemma}\label{ScalarThetaNewformLevel2}
If $f(\tau)=\sum\limits_{n\geq 1}a_f(n)q_{\tau}^n \in\mathcal{S}_{2m+2}^{\text{new}}\left(\Gamma_0(2)\right)$ is a normalized newform, then the following are true.

\noindent
(1)
 $$
\langle \pi_{\hol}([\mathcal{H},\theta]_{m}),f\rangle=-\frac{1}{2}\cdot\frac{{\binom{2m}{m}}}{4^m}\langle f,f \rangle. 
$$

\noindent
(2) $$
\langle \pi_{\hol}([\mathcal{H},\theta]_{m}),f|V(2)\rangle=0. 
$$
\end{lemma}

\begin{proof}
The same argument in the proof of Lemma~\ref{ScalarThetaNewformLevel4} gives us that 
$$
\langle [\mathcal{H},\theta]_m, f\rangle=-\frac{2}{3}\cdot\frac{{\binom{2m}{m}}}{4^m}\langle f,f\rangle-\frac{1}{24}\langle [E_{\frac{3}{2}},\theta]_m, f|W_4\rangle.
$$
To emulate the rest of the argument, we need to determine $f|W_4.$ To this end, note that $f|W_4=f|W_2V(2)=\lambda_2(f)\cdot f|V(2).$ This implies that 
$$
(f|W_4)(\tau)=\lambda_2(f)\cdot 2^{m+1}\sum\limits_{n\geq 1} a_f\left(\frac{n}{2}\right)q_{\tau}^n.
$$
Therefore, we have that
\begin{align}
\langle [E_{\frac{3}{2}},\theta]_m, f|W_4\rangle \notag&=\frac{1}{6}{\binom{m+\frac{1}{2}}{m}}\cdot\frac{(2m)!}{(4\pi)^{2m+1}}\cdot 2^{m+1}\lambda_2(f)\cdot\sum\limits_{n\geq 1}\frac{2\cdot a_f\left(\frac{n^2}{2}\right)}{n^{2m+2}}\\
\notag&=\frac{1}{6}{\binom{m+\frac{1}{2}}{m}}\cdot\frac{(2m)!}{(4\pi)^{2m+1}}\cdot 2^{m+2}\lambda_2(f)\cdot\sum\limits_{n\geq 1}\frac{a_f(2n^2)}{2^{2m+2}n^{2m+2}}\\
\notag&=\frac{1}{6}{\binom{m+\frac{1}{2}}{m}}\cdot\frac{(2m)!}{(4\pi)^{2m+1}}\cdot \lambda_2(f)\cdot\frac{a_f(2)}{2^{m}}\sum\limits_{n\geq 1}\frac{a_f(n^2)}{n^{2m+2}}\\
\notag&=-\frac{1}{6}{\binom{m+\frac{1}{2}}{ m}}\cdot\frac{(2m)!}{(4\pi)^{2m+1}}\sum\limits_{n\geq 1}\frac{a_f(n^2)}{n^{2m+2}}.
\end{align}
The proof of Lemma~\ref{ScalarThetaNewformLevel4} then shows that
$$
\langle \pi_{\hol}([\mathcal{H},\theta]_{m}),f\rangle=-\frac{1}{2}\cdot\frac{{\binom{2m}{m}}}{4^m}\langle f,f \rangle. 
$$
The proof of (2) is similar.
\end{proof}

If $f$ is of level $1,$ we have the following lemma.

\begin{lemma}\label{ScalarThetaNewformLevel1}
If $f(\tau)=\sum\limits_{n\geq 1}a_f(n)q_{\tau}^n \in\mathcal{S}_{2m+2}(\SL_2(\Z))$ is a normalized newform, then the following are true.

\noindent
(1) $$
\langle \pi_{\hol}([\mathcal{H},\theta]_{m}),f\rangle=-\frac{1}{2}\cdot\frac{{\binom{2m}{m}}}{4^m}\langle f,f \rangle. 
$$

\noindent
(2) $$
\langle \pi_{\hol}([\mathcal{H},\theta]_{m}),f|V(2)\rangle=-\frac{1}{6}\cdot\frac{a_f(2)}{2^{m}}\cdot\frac{{\binom{2m}{m}}}{4^m}\langle f,f \rangle. 
$$

\noindent
(3) $$
\langle \pi_{\hol}([\mathcal{H},\theta]_{m}),f|V(4)\rangle=-\frac{1}{2}\cdot\frac{{\binom{2m}{m}}}{4^m}\langle f,f \rangle. 
$$
\end{lemma}

\begin{proof}

The proofs of (1) and (3) are similar to the proof of Lemma~\ref{ScalarThetaNewformLevel4}. The proof of (2) is similar to the proof of Lemma~\ref{ScalarThetaNewformLevel2} with the following modification. 

\begin{align}
\sum\limits_{n\geq 1}\frac{1}{2^{m+1}}\frac{a_f(2n^2)}{n^{2m+2}}\notag&=\sum\limits_{n\geq 1}\frac{1}{2^{m+1}}\frac{1}{n^{2m+2}}\cdot\left[a_f(2)a_f(n^2)-2^{2m+1}a_f\left(\frac{n^2}{2}\right)\right]\\
\notag&=\frac{a_f(2)}{2^{m+1}}\sum\limits_{n\geq 1}\frac{a_f(n^2)}{n^{2m+2}}-\frac{1}{2}\sum\limits_{n\geq 1}\frac{1}{2^{m+1}}\frac{a_f(2n^2)}{n^{2m+1}}\\
\notag&=\frac{2}{3}\cdot\frac{a_f(2)}{2^{m+1}}\sum\limits_{n\geq 1}\frac{a_f(n^2)}{n^{2m+2}}.
\end{align}
\end{proof}

We now apply Lemmas~\ref{ScalarThetaNewformLevel4} through \ref{ScalarThetaNewformLevel1} and the results of Section~\ref{SectionPetersson} to determine the decomposition of $\pi_{\hol}([\mathcal{H}(\tau),\theta(\tau)]_m).$

\begin{proposition}\label{ThetaCoefficients}
Consider the decomposition
$$
\pi_{\hol}([\mathcal{H}(\tau),\theta(\tau)]_{m})=\sum c_1(f)f(\tau) + \sum c_2(f)(f|V(2))(\tau)+\sum c_4(f)(f|V(4))(\tau),
$$
where the sums go over normalized newforms of levels dividing $4$ and $c_i(f)\in\C$ for $i\in\{1,2,4\}.$ Then the following are true.

\noindent
(1) If $f\in\mathcal{S}_{2m+2}^{\text{new}}\left(\Gamma_0(4)\right),$ then 
$$
c_1(f)=-\frac{1}{3}\cdot\frac{{\binom{2m}{m}}}{4^m}.
$$

\noindent
(2) If $f\in\mathcal{S}_{2m+2}^{\text{new}}\left(\Gamma_0(2)\right),$ then 
$$
c_1(f)=-\frac{2}{3}\cdot\frac{{\binom{2m}{m}}}{4^m};\ \ \ \ \ \ \ \ \ \ \ \ \ \ \ \ \ \ \ \ \ \ \ \ \ c_2(f)=-\frac{\lambda_2(f)}{3}\cdot\frac{{{2m}\choose m}}{4^m}.
$$

\noindent
(3) If $f\in\mathcal{S}_{2m+2}\left(\SL_2(\Z)\right),$ then 
$$
c_1(f)=-\frac{{{2m}\choose m}}{4^m};\ \ \ \ \ \ \ \ \ \ \ \ c_2(f)=0;\ \ \ \ \ \ \ \ \ \ \ \ c_4(f)=-\frac{{{2m}\choose m}}{4^m}.
$$
\end{proposition}

\begin{proof}
The proofs of (1), (2), and (3) are similar. Therefore, we only prove (2) for brevity. If we take inner products with $f$ and $f|V(2)$ on both sides of the decomposition, then  Lemma~\ref{ScalarThetaNewformLevel2} gives that
$$
c_1(f)\langle f,f\rangle+c_2(f)\langle f,f|V(2)\rangle=-\frac{1}{2}\cdot\frac{{{2m}\choose m}}{4^m}\langle f,f\rangle
$$ 
and 
$$
c_1(f)\langle f,f|V(2)\rangle+c_2(f)\langle f|V(2),f|V(2)\rangle=0.
$$
Lemma~\ref{VOperatorNewformLevel2} then gives that
$$
c_1(f)-\frac{\lambda_2(f)}{2}c_2(f)=-\frac{1}{2}\cdot\frac{{{2m}\choose{m}}}{4^m}
$$
and
$$
-\frac{\lambda_2(f)}{2}c_1(f)+c_2(f)=0.
$$
We obtain the claim by solving the system of equations.
\end{proof}

\begin{remark}
For our purposes, we only need the coefficients $c_1(f).$ 
\end{remark}

\subsection{Decomposition of $\pi_{\hol}([\mathcal{H}(\tau),\theta(4\tau)]_{m})$}

Here we decompose $\pi_{\hol}([\mathcal{H}(\tau),\theta(4\tau)]_m).$  Since $\theta(4\tau)$ is modular on $\Gamma_0(16),$ we use  Proposition~\ref{UnfoldingInGeneral2} instead of \ref{UnfoldingInGeneral}. In the notation of Proposition~\ref{UnfoldingInGeneral2}, we first describe the Fourier series of $g|A$ where $g(\tau)=\theta(4\tau)$ and $A\in S.$ Namely, we have the following lemma.

\begin{lemma}\label{Theta4Lemma}
If $g(\tau)=\theta(4\tau)$ and $q_{\tau}=e^{2\pi i\tau},$ then the following are true.

\noindent
(1) 
$$
\left(g\Big|\begin{pmatrix}
1 & 0 \\
4 & 1 
\end{pmatrix}\right)(\tau)=\frac{1-i}{2}\left(\sum\limits_{\substack{n\in\Z \\ 2\mid n}} q_{\tau}^{n^2}+i\sum\limits_{\substack{n\in\Z \\ 2\nmid n}}q_{\tau}^{n^2}\right).
$$

\noindent
(2) 
$$
\left(g\Big|\begin{pmatrix}
3 & -1 \\
4 & -1 
\end{pmatrix}\right)(\tau)=\frac{1+i}{2}\left(\sum\limits_{\substack{n\in\Z \\ 2\mid n}} q_{\tau}^{n^2}-i\sum\limits_{\substack{n\in\Z \\ 2\nmid n}}q_{\tau}^{n^2}\right).
$$

\noindent
(3) 
$$
\left(g\Big|\begin{pmatrix}
1 & 0 \\
8 & 1 
\end{pmatrix}\right)(\tau)=\sum\limits_{\substack{n\in\Z \\ 2\nmid n}} q_{\tau}^{n^2}.
$$
\end{lemma}

\begin{proof}
The proofs of (1),(2), and (3) are analogous. Therefore, we only prove (1) for brevity. By definition of $V(4),$ we have that
\begin{align}
g\Big|\begin{pmatrix}
1 & 0 \\
4 & 1
\end{pmatrix}\notag&= 4^{-\frac{k}{2}}\theta|V(4)\begin{pmatrix}
1 & 0 \\
4 & 1
\end{pmatrix}\\
\notag&=4^{-\frac{k}{2}}\theta\Big|\begin{pmatrix}
1 & 0 \\
1 & 1
\end{pmatrix}V(4)\\
\notag&=4^{-\frac{k}{2}}\theta\Big|\begin{pmatrix}
1 & 1 \\
0 & 1
\end{pmatrix}\begin{pmatrix}
0 & -1 \\
1 & 0
\end{pmatrix}\begin{pmatrix}
1 & 1 \\
0 & 1
\end{pmatrix}|V(4).
\end{align}
Since $\theta$ is clearly periodic with period $1,$ we have that $\theta\Big|\begin{pmatrix}
1 & 1 \\
0 & 1
\end{pmatrix}=\theta.$  Furthermore, by the modularity properties of $\theta$ (see Corollary 2.3.21 of \cite{CohenStromberg}), we have that
$$
\left(\theta\Big\vert\begin{pmatrix}
0 & -1\\
1 & 0
\end{pmatrix}\right)(\tau)=\frac{1}{\sqrt{2i}}\theta\left(\frac{\tau}{4}\right).
$$
Putting all this together, we have
$$
\left(g\Big|\begin{pmatrix}
1 & 0 \\
4 & 1
\end{pmatrix}\right)(\tau)=\frac{1-i}{2}\sum\limits_{n\in\Z} i^{n^2}q_{\tau}^{n^2},
$$
which is exactly (1).
\end{proof}

The following three lemmas determine the inner product of $\pi_{\hol}([\mathcal{H}(\tau),\theta(4\tau)]_m)$ with elements of $\{f|V(d)\}$ where $f$ is a newform of level $N|4$ and $dN|16.$ The first lemma concerns the level $4.$

\begin{lemma}\label{ScalarTheta4NewformLevel4}
If $f(\tau)=\sum\limits_{n\geq 1}a_f(n)q_{\tau}^n \in\mathcal{S}_{2m+2}^{\text{new}}\left(\Gamma_0(4)\right)$ is a normalized newform, then the following are true.

\noindent
(1) $$
\langle \pi_{\hol}([\mathcal{H}(\tau),\theta(4\tau)]_{m}),f\rangle=-\frac{1}{6}\cdot\frac{{{2m}\choose{m}}}{4^m}\langle f,f \rangle. 
$$

\noindent
(2) $$
\langle \pi_{\hol}([\mathcal{H}(\tau),\theta(4\tau)]_{m}),f|V(2)\rangle=0.
$$

\noindent
(3) $$
\langle \pi_{\hol}([\mathcal{H}(\tau),\theta(4\tau)]_{m}),f|V(4)\rangle=\frac{1}{6}\cdot\frac{{{2m}\choose{m}}}{4^m}\langle f,f \rangle.
$$
\end{lemma}

\begin{proof}

As in the proof of Lemma~\ref{ScalarThetaNewformLevel4}, we have
\begin{align}
\langle f,\pi_{\hol}([\mathcal{H}(\tau),\theta(4\tau)]_m)\rangle\notag&=\langle f, [\mathcal{H}(\tau),\theta(4\tau)]_m\rangle \\
\notag&=-\frac{1}{12}\langle f,[E_{\frac{3}{2}}(\tau),\theta(4\tau)]_m\rangle -\frac{1}{12}\langle f, (1-i)2^{-3/2} [F_{\frac{3}{2}}(\tau),\theta(4\tau)]_m\rangle.
\end{align}
Proposition~\ref{UnfoldingInGeneral2} shows that
$$
\langle f, [E_{\frac{3}{2}}(\tau),\theta(4\tau)]_m \rangle=\frac{1}{24}{{m-\frac{1}{2}}\choose m}\frac{(2m)!}{(4\pi)^{2m+1}}\sum\limits_{A\in S}\sum\limits_{n\geq 1}\frac{a_{f,A}(n)\overline{a_{\theta(4\tau),A}(n)}}{n^{m+1}}.
$$
Since $f$ is modular on $\Gamma_0(4),$ $a_{f,A}(n)=a_f(n)$ for all $A\in S.$ Therefore, Lemma~\ref{Theta4Lemma} gives that
$$
\langle f,[E_{\frac{3}{2}}(\tau),\theta(4\tau)]_m\rangle=\frac{1}{6}{{m+\frac{1}{2}}\choose m}\cdot\frac{(2m)!}{(4\pi)^{2m+1}}\sum\limits_{n\geq 1}\frac{a_f(n^2)}{n^{2m+2}}.
$$
An application of Lemma~\ref{IntermediateUnfolding} then shows that
$$
\frac{-1}{12}\langle f,[E_{\frac{3}{2}}(\tau),\theta(4\tau)]_m\rangle=-\frac{1}{3}\cdot\frac{{{2m}\choose m}}{4^m}\langle f,f\rangle.
$$
To finish the proof of (1), it remains to compute $\langle f,[F_{\frac{3}{2}}(\tau),\theta(4\tau)]_m\rangle$. To this end, part (2) of Lemma~\ref{FrickeProps} gives
$$
\langle f,[F_{\frac{3}{2}}(\tau),\theta(4\tau)]_m \rangle=\langle f|W_{16}, 2^{3/2}[G_{\frac{3}{2}},\theta(4\tau)|W_{16}]_m\rangle. 
$$
Applying part (2) of Lemma~\ref{FrickeProps}, we have
$$
\langle f,[F_{\frac{3}{2}}(\tau),\theta(4\tau)]_m\rangle=\frac{1}{\sqrt{2}}e^{i\frac{\pi}{4}}\langle f|W_{16},[G_{\frac{3}{2}}(\tau),\theta(\tau)]_m\rangle.
$$
Since $f$ is a newform of level $4,$ we have that  $f|W_{16}=f|W_{4}|V(4)=-f|V(4).$  An application of Proposition~\ref{UnfoldingInGeneral} with $G_{\frac{3}{2}}$ and $\Gamma_0(16)$ in place of $E_{\frac{3}{2}}$ and $\Gamma_0(4)$ shows that
\begin{align}
\langle f|W_{16},[G_{\frac{3}{2}},\theta]_m\rangle\notag&=-\frac{1}{12}{{m+\frac{1}{2}}\choose m}\cdot\frac{(2m)!}{(4\pi)^{2m+1}}\sum\limits_{n\geq 1}\frac{a_f(n^2)}{n^{2m+2}}.
\end{align}
Putting all this together, another application of Lemma~\ref{IntermediateUnfolding} gives (1).

We sketch the proof of (2). Since $f|V(2)\in\mathcal{S}_{2m+2}\left(\Gamma_0(8)\right),$ the Fourier series of $f|V(2)\Big|\begin{pmatrix}
1 & 0 \\
8 & 1
\end{pmatrix}$ is equal to the Fourier series of $f|V(2).$  On the other hand, we have
$$
f|V(2)\Big|\begin{pmatrix}
1 & 0\\
4 & 1
\end{pmatrix}=f\Big|\begin{pmatrix}
1 & 0 \\
2 & 1
\end{pmatrix}|V(2).
$$
It is easy to verify that $\begin{pmatrix}
1 & 0 \\
2 & 1
\end{pmatrix}$ normalizes $\Gamma_0(4)$ and $\Gamma_p(4)$ for all odd primes $p$.  Therefore, Theorem~\ref{Mult1} implies that $f\Big|\begin{pmatrix}
1 & 0 \\
2 & 1
\end{pmatrix}=\pm f.$ Since $f$ is a newform on $\Gamma_0(4)$ and $\left|\Gamma_0(4)/\Gamma_0(2)\right|=2,$ we have that $f\Big|\begin{pmatrix}
1 & 0 \\
2 & 1
\end{pmatrix}=-f.$ Furthermore, we have that
$$
f|V(2)\Big|\begin{pmatrix}
3 & -1 \\
4 & -1
\end{pmatrix}=f\Big|\begin{pmatrix}
7 & -2 \\
4 & -1
\end{pmatrix}|V(2)\Big|\begin{pmatrix}
1 & 0 \\
4 & 1
\end{pmatrix}=f|V(2)\Big|\begin{pmatrix}
1 & 0 \\
4 & 1
\end{pmatrix}.
$$
The computation of $\langle f|V(2),[E_{\frac{3}{2}}(\tau),\theta(4\tau)]_m\rangle$ now follows exactly as in the proof of (1). To evaluate $\langle f|V(2), [F_{\frac{3}{2}}(\tau),\theta(4\tau)]_m\rangle,$ we must evaluate $f|V(2)|W_{16}.$ A direct computation gives that
$$
f|V(2)|W_{16}=f\Big|\begin{pmatrix}
2 & 0 \\
0 & 2
\end{pmatrix}|W_4|V(2)=-f|V(2).
$$
The remainder of the proof is analogous to the proof of (1). The proof of (3) is similar.
\end{proof}

If $f$ is of level $2,$ we have the following lemma.

\begin{lemma}\label{ScalarTheta4NewformLevel2}
If $f(\tau)=\sum\limits_{n\geq 1}a_f(n)q_{\tau}^n \in\mathcal{S}_{2m+2}^{\text{new}}\left(\Gamma_0(2)\right)$ is a normalized newform, then the following are true.

\noindent
(1) $$
\langle \pi_{\hol}([\mathcal{H}(\tau),\theta(4\tau)]_{m}),f\rangle=-\frac{1}{4}\cdot\frac{{{2m}\choose{m}}}{4^m}\langle f,f \rangle. 
$$

\noindent
(2) $$
\langle \pi_{\hol}([\mathcal{H}(\tau),\theta(4\tau)]_{m}),f|V(2)\rangle=0. 
$$

\noindent
(3) $$
\langle \pi_{\hol}([\mathcal{H}(\tau),\theta(4\tau)]_{m}),f|V(4)\rangle=0. 
$$

\noindent
(4) $$
\langle \pi_{\hol}([\mathcal{H}(\tau),\theta(4\tau)]_{m}),f|V(8)\rangle=-\frac{\lambda_2(f)}{4}\cdot\frac{{{2m}\choose{m}}}{4^m}\langle f,f \rangle. 
$$
\end{lemma}

\begin{proof}

The proofs of (1), (2) and (3) are the analogous to the proof of Lemma~\ref{ScalarTheta4NewformLevel4}.

The proof of (4) is similar, but the Fourier series of $f|V(8)$ at cusps is more involved. The Fourier series at $\frac{1}{8}$ is computed as follows.

\begin{align}
\left(f|V(8)\Big|\begin{pmatrix}
1 & 8 \\
0 & 1
\end{pmatrix}\right)(\tau)\notag&=\left(f|W_2\cdot V(4)\begin{pmatrix}
1 & 0 \\
-8 & 1 
\end{pmatrix}\begin{pmatrix}
1 & \frac{1}{8} \\
0 & 1
\end{pmatrix}\right)(\tau)\\
\notag&=\lambda_2(f) \left(f|V(4)\Big|\begin{pmatrix}
1 & \frac{1}{8} \\
0 & 1
\end{pmatrix}\right)(\tau)\\
\notag&= \lambda_2(f)\cdot 4^{m+1}\sum\limits_{n\geq 1} a_f(n)(-1)^n q_\tau^{4n}.
\end{align}
Similarly, we have that
\begin{align}
\left(f|V(8)\Big|\begin{pmatrix}
1 & 0 \\
4 & 1
\end{pmatrix}\right)(\tau)\notag&=\left(f\Big|\begin{pmatrix}
2 & 0 \\
0 & 2
\end{pmatrix} W_2 \begin{pmatrix}
1 & 0 \\
-4 & 1 
\end{pmatrix}\begin{pmatrix}
1 & \frac{1}{4} \\
0 & 1
\end{pmatrix}\right)(\tau)\\ 
\notag&=\lambda_2(f)\sum\limits_{n\geq 1} a_f(n)\cdot i^n q_{\tau}^n.
\end{align}
and
$$
\left(f|V(8)\Big|\begin{pmatrix}
3 & -1 \\
4 & -1
\end{pmatrix}\right)(\tau)=\lambda_2(f)\sum\limits_{n\geq 1} a_f(n)\cdot (-i)^n q_{\tau}^n.
$$
\end{proof}

If $f$ is of level $1,$ we have the following lemma.

\begin{lemma}\label{ScalarTheta4NewformLevel1}
If $f(\tau)=\sum\limits_{n\geq 1}a_f(n)q_{\tau}^n \in\mathcal{S}_{2m+2}(\SL_2(\Z))$ is a normalized newform, then the following are true.

\noindent
(1) If $l\in\{0,2,4\},$ then
$$
\langle \pi_{\hol}([\mathcal{H}(\tau),\theta(4\tau)]_{m}),f|V(2^l)\rangle=-\frac{1}{4}\cdot\frac{{{2m}\choose{m}}}{4^m}\langle f,f \rangle.
$$

\noindent
(2) If $l\in\{1,3\},$ then
$$
\langle \pi_{\hol}([\mathcal{H}(\tau),\theta(4\tau)]_{m}),f|V(2^l)\rangle=-\frac{1}{6}\cdot\frac{a_f(2)}{2^{m+1}}\frac{{{2m}\choose{m}}}{4^m}\langle f,f \rangle. 
$$
\end{lemma}

\begin{proof}
The proof follows similar lines to the proof of Lemmas~\ref{ScalarTheta4NewformLevel4} and ~\ref{ScalarTheta4NewformLevel2}. The only modification is
$$
\sum\limits_{n\geq 1}\frac{1}{2^{m+1}}\cdot\frac{a_f(2n^2)}{n^{2m+2}}=\frac{2}{3}\cdot\frac{a_f(2)}{2^{m+1}}\sum\limits_{n\geq 1}\frac{a_f(n^2)}{n^{2m+2}},
$$
which was shown in the proof of Lemma~\ref{ScalarThetaNewformLevel1}.
\end{proof}

The previous results yield a partial decomposition of $\pi_{\hol}([\mathcal{H}(\tau),\theta(4\tau)]_m).$ 

\begin{proposition}\label{ThetaCoefficients4}
Consider the decomposition
\begin{align}
\pi_{\hol}([\mathcal{H}(\tau),\theta(4\tau)]_{{m}})\notag&=\sum c_1(f)f(\tau) + \sum c_2(f)(f|V(2))(\tau)+\sum c_4(f)(f|V(4))(\tau)
\\
\notag&+\sum c_8(f)(f|V(8))(\tau)+\sum c_{16}(f)(f|V(16))(\tau),
\end{align}
where the sums go over newforms of levels dividing $16$ and $c_i(f)\in\C$ for $i\in\{1,2,4,8,16\}.$  The following are true.

\noindent
(1) If $f\in\mathcal{S}_{2{m}+2}^{\text{new}}\left(\Gamma_0(4)\right),$ then we have
$$
c_1(f)=-\frac{1}{6}\cdot\frac{{{2{m}}\choose{{m}}}}{4^{m}}; \qquad  \qquad  \quad c_2(f)=0; \qquad \qquad \quad c_4(f)=\frac{1}{6}\cdot\frac{{{2{m}}\choose{{m}}}}{4^{m}}.
$$

\noindent
(2) If $f\in\mathcal{S}_{2{m}+2}^{\text{new}}\left(\Gamma_0(2)\right),$ then we have 
$$
c_1(f)=-\frac{1}{3}\cdot\frac{{{2{m}}\choose{{m}}}}{4^{m}}; \qquad  c_2(f)=-\frac{\lambda_2(f)}{6}\cdot\frac{{{2{m}}\choose{{m}}}}{4^{m}}; \qquad 
c_4(f)=-\frac{1}{6}\cdot\frac{{{2{m}}\choose{{m}}}}{4^{m}}.
$$

\noindent
(3) If $f(\tau)=\sum\limits_{n\geq 1}a_f(n)q_{\tau}^n \in\mathcal{S}_{2{m}+2}(\SL_2(\Z)),$ then we have
$$
c_1(f)=-\frac{1}{2}\cdot\frac{{{2{m}}\choose{{m}}}}{4^{m}}; \qquad 
c_2(f)=\frac{1}{2}\cdot\frac{a_f(2)}{2^{{m}+1}}\frac{{{2{m}}\choose{{m}}}}{4^{m}}; \qquad c_4(f)=-\frac{3}{4}\cdot\frac{{{2{m}}\choose{{m}}}}{4^{m}}.
$$
\end{proposition}

\begin{proof}
The proof is analogous to the proof of Proposition~\ref{ThetaCoefficients}.
\end{proof}

\subsection{Bounds on the Fourier coefficients}

Here we apply Deligne's bound to the decompositions above.
 
\begin{lemma}\label{BoundCoefficients}
If $p\geq 5$ is a prime and ${m}\in\N$, then the following are true.

\noindent
(1) If
$
\pi_{\hol}([\mathcal{H}(\tau),\theta(\tau)]_{m})=\sum\limits_{n\geq 1} a_m(n)q_{\tau}^n,
$
then
$$
|a_m(p)|\leq\frac{2}{3}\cdot\frac{{{2{m}}\choose {m}}}{4^{m}}\cdot({m}-1)\cdot p^{{m}+1/2}.
$$

\noindent
(2) If $
\pi_{\hol}([\mathcal{H}(\tau),\theta(4\tau)]_{m})=\sum\limits_{n\geq 1}b_m(n)q_\tau^n,
$
then
$$
|b_m(4p)|\leq \frac{4}{3}({m}-1)\cdot{{2{m}}\choose {m}}p^{{m}+1/2}.
$$
\end{lemma}

\begin{proof}

To prove (1), note that only newforms contribute to $a_m(p).$ In other words,
$$
a_m(p)=-\frac{1}{3}\cdot\frac{{{2{m}}\choose {m}}}{4^{m}}\left[\sum\limits_{f}a_f(p)+2\sum\limits_g a_g(p)+3\sum\limits_{h}a_h(p)\right],
$$
where $f(\tau)=\sum\limits_{n\geq 1}a_f(n)q_{\tau}^n,g(\tau)=\sum\limits_{n\geq 1}a_g(n)q_{\tau}^n$ and $h(\tau)=\sum\limits_{n\geq 1}a_h(n)q_{\tau}^n$ run over normalized newforms of levels $4,2$ and $1$ respectively. One can easily see that 
$$
\dim\mathcal{S}^{\text{new}}_{2{m}+2}(\Gamma_0(4))+2\dim\mathcal{S}^{\text{new}}_{2{m}+2}(\Gamma_0(2))+3\dim\mathcal{S}_{2{m}+2}(\Gamma_0(1))=({m}-1).
$$ A direct application of Theorem~\ref{DeligneTheorem} then gives
$$
|a_m(p)|\leq\frac{1}{3}\cdot\frac{{{2{m}}\choose {m}}}{4^{m}}\cdot({m}-1) d(p)\cdot p^{{m}+1/2},
$$
where $d(\cdot)$ is the divisor function. This shows (1).

We prove (2) in analogous manner. In this case, images of newforms under $V(2)$ and $V(4)$ contribute to $b_m(4p).$ More precisely,
\begin{align}
b_m(4p)\notag&=-\frac{1}{12}\cdot\frac{{{2{m}}\choose{{m}}}}{4^{m}}\bigg[-2\sum\limits_f 4^{{m}+1}a_f(p)\\ 
\notag&+4\sum\limits_{g} 4^{{m}}a_g(p)-2\sum\limits_g 2^{m}\cdot 2^{{m}+1}a_g(p)+2\sum\limits_g 4^{{m}+1}a_g(p)\\ 
\notag&+6\sum\limits_{h} a_h(4p)-6\sum\limits_h a_h(2)\cdot a_h(2p)+9\sum\limits_h 4^{{m}+1}a_h(p)\bigg],
\end{align}
where $f,g$ and $h$ are as above. Using the Hecke relations for Fourier coefficients of $h,$ we get
$$
b_m(4p)=-\frac{1}{12}\cdot\frac{{{2{m}}\choose{{m}}}}{4^{m}}\left[-2\sum\limits_f 4^{{m}+1}a_f(p)+2\sum\limits_g 4^{{m}+1}a_g(p)+6\sum\limits_h 4^{{m}+1}a_h(p)\right].
$$
A direct application of Theorem~\ref{DeligneTheorem} as above then gives (2).
\end{proof}

\section{Determining Distributions}\label{SectionProbTheory}

To determine an explicit $O(3)$ distribution for the values $A_{\lambda},$ we make use of explicit circular distributions of the traces $a_\lambda^{\Cl}(p)$ and interpret the $O(3)$ distribution in terms of the semicircular distribution.

\subsection{Beurling-Selberg and Chebyshev Polynomials}\label{SubsAnalyticTools}
To explicitly determine semicircular distribution from moments, we make use of the Beurling-Selberg trigonometric polynomials and Chebyshev polynomials. Here we recall the definitions and properties of these polynomials.

To determine semicircular distributions, we approximate indicator functions by trigonometric polynomials due to Beurling and Selberg. To make this precise, fix a closed interval $I=[a,b]\subset[0,1]$ and let $\chi_I(x)$ be the indicator function of $I.$ Then there exists trigonometric polynomials $S_M^{+}$ and $S_M^{-}$ of degree at most $M,$ such that the following hold (see \cite{Murty2}).

\noindent
(1) $S_M^{-}(x)\leq \chi_I(x)\leq S_M^+(x)$ for all $x\in[0,1].$ 

\noindent
(2) If $\hat{S}^{\pm}_M(m)$ is the $m$-th Fourier coefficient of $S^{\pm}_M,$ then
$$
\hat{S}^{\pm}_M(0)=b-a\pm\frac{1}{M+1}
$$
Furthermore, if $e(x):=e^{2\pi ix}$ and $0<m\leq|M|,$ then
$$
\left|\hat{S}^{\pm}_M(m)-\frac{e(ma)-e(mb)}{2\pi im}\right|\leq \frac{1}{M+1}.
$$
The polynomials $S_M^+$ and $S_M^{-}$ are respectively called the majorant and minorant Selberg polynomials. To make use of explicit moment estimates in our use of Selberg polynomials, we express $\cos(m\theta)$ as a polynomial in $\cos(\theta).$ To this end, define the Chebyshev polynomials by the recurrence relation
\begin{align*}
&U_0(x)=1 \\
&U_1(x)=2x \\
&U_m(x)=2xU_{m-1}(x)-U_{m-2}(x) \text{ for }m\geq 2.
\end{align*}
An induction argument on $m$ shows that
$$
\cos(m\theta)=\frac{1}{2}\left(U_m(\cos\theta)-U_{m-2}(\cos\theta)\right),
$$
for all $m\geq 2.$ Furthermore, it turns out that the coefficients of these Chebyshev polynomials encode the Rankin-Cohen brackets of modular forms of weight $\frac{3}{2}$ with univariate theta functions of weight $\frac{1}{2}.$  To make this precise, if $m$ is a nonnegative integer, then we write $U_{m}(x)=\sum\limits_{l=0}^{m} b(l;m)x^{l}.$ In this notation, we have the following proposition.

\begin{proposition}\label{ChebyshevProp}
If $f(\tau)=\sum\limits_{n\geq 0}a_f(n)q_{\tau}^n$ and $t\in\N,$ then
$$
[f(\tau),\theta(t\tau)]_m=\sum\limits_{n\geq 1}c_m(n)q_{\tau}^n,
$$
where\footnote{We define $a_f(n)=0$ for $n<0.$}
$$
c_m(n)=\frac{{{2m}\choose m}}{4^m}\cdot n^m\sum\limits_{l=0}^{m}b(2l;2m)\left[\sum\limits_{s\in\Z}a_f(n-ts^2)\left(\sqrt{\frac{t}{n}}s\right)^{2l}\right],
$$
and the Rankin-Cohen bracket is of weight $(\frac{3}{2},\frac{1}{2}).$
\end{proposition}

\begin{proof}

By definition of the Rankin-Cohen bracket, the Fourier coefficient $c_m(n)$ is given by
$$
c_m(n)=(-n)^m\sum\limits_{l=0}^m\left[\sum\limits_{r=0}^l(-1)^l{{m+\frac{1}{2}}\choose r} {{m-\frac{1}{2}}\choose m-r} {{m-r}\choose {l-r}}\right]\sum\limits_{s\in\Z} a_f(n-ts^2)\left(\sqrt{\frac{t}{n}}s\right)^{2l}.
$$
Since 
$
{{m-\frac{1}{2}}\choose m}=\frac{{{2m}\choose m}}{4^m}
$
for positive integers $m,$ the claim reduces to showing that
\begin{equation}\label{BracketChebEq}
\sum\limits_{r=0}^l(-1)^l{{m+\frac{1}{2}}\choose r} {{m-\frac{1}{2}}\choose m-r} {{m-r}\choose {l-r}}=(-1)^m{{m-\frac{1}{2}}\choose m}b(2l;2m),
\end{equation}
for all $0\leq l\leq m.$ To show (\ref{BracketChebEq}), define
$$
A_r=(-1)^l{{m+\frac{1}{2}}\choose r} {{m-\frac{1}{2}}\choose m-r} {{m-r}\choose {l-r}},
$$
for each $0\leq r\leq l.$ The sum $\sum\limits_{r=0}^l A_r$ can be computed as the special value of a classical hypergeometric series (for background, see \cite{bailey}). Namely,
$$
\frac{A_{r+1}}{A_r}=\frac{(r-l)(r-(m+1/2))}{(r+1)(r+1/2)}
$$
is a rational function in $r,$ and therefore
$$
\sum\limits_{r=0}^l A_r=(-1)^l{{m-\frac{1}{2}}\choose m}{m \choose l}{_2F_1}\left(\begin{array}{cc}
-\frac{1}{2}-m, & -l\\
~& \frac{1}{2}
\end{array}\mid 1\right),
$$
where ${_2F_1}$ is a classical hypergeometric function. By Gauss's identity (see (1.3) of \cite{bailey}), we deduce that
$$
\sum\limits_{r=0}^l A_r=(-1)^l{{m-\frac{1}{2}}\choose m}{m \choose l}\frac{\Gamma(1/2)\Gamma(l+m+1)}{\Gamma(m+1)\Gamma(l+1/2)},
$$
where $\Gamma$ is the standard Gamma function (for definition and properties of the standard Gamma function, see III.1.1 of \cite{mertensPhD}). Since $\Gamma(\frac{1}{2})=\sqrt{\pi},$ Legendre's duplication formula shows that
$$
\frac{\Gamma(1/2)}{l!\Gamma(l+1/2)}=\frac{2^{2l-1}}{l\cdot(2l-1)!}.
$$
It remains to show that 
$$
b(2l;2m)=\frac{(-1)^{m-l}2^{2l-1}(l+m)!}{l\cdot(m-l)!(2l-1)!}.
$$
This is shown by an induction on $m$ using the recurrence relations $b(l;m)=2b(l-1;m)-b(l;m-2)$ for all $l\geq 1$ and $b(0;m)=-b(0;m-2).$
\end{proof}

Finally, we bound the magnitude of $U_{2m}(x)$ on $[-1,1].$ To this end, we have the following lemma.

\begin{lemma}\label{BoundChebyshev}
If $m\in\N$ and $x\in[-1,1],$ then $|U_{m}(x)|\leq m+1.$
\end{lemma}

\begin{proof}
An induction argument shows\footnote{In fact, this is an alternative definition of the Chebyshev polynomials $U_{m}(x).$} that $U_{m}(\cos\theta)=\frac{\sin (m+1)\theta}{\sin\theta}$ for all $\theta\in(0,\pi).$ The claim follows by an elementary computation.
\end{proof}

\subsection{Circular Distribution of Clausen Curves}\label{SectionCircular}

Here we obtain certain explicit circular distributions of the traces $a_{\lambda}^{\Cl}.$  The following proposition explicitly determines the distribution of the absolute value of these traces.

\begin{proposition}\label{EffectiveClausenSym}
If $0\leq a<b\leq 1$ and $p\geq 5,$ then we have
$$
\left|\frac{N(a,b;p)}{p}-2\mu_{\ST}([a,b])\right|\leq\frac{26.52}{p^{1/4}},
$$
where  $N(a,b; p):=\sum\limits_{\lambda\in\F_p\setminus\{0,-1\}}\chi_{[a,b]}(|a_\lambda^{\Cl}(p)|)$ and $\mu_{\ST}(I):=\frac{2}{\pi}\int_I\sqrt{1-x^2}\ dx$ denotes the semicircular measure on $[-1,1].$
\end{proposition}

\begin{proof}

Following \cite{murty}, we define $\theta_a,\theta_b\in[0,\frac{\pi}{2}]$ such that $\cos\theta_a=a, \cos\theta_b=b$ and define $I'=[\frac{\theta_b}{\pi},\frac{\theta_a}{\pi}].$  Furthermore, fix a prime $p\geq 5$ and write $a_{\lambda}^{\Cl}(p)=2\sqrt{p}\cos\theta_\lambda,$ where $\theta_\lambda\in[0,\pi].$ In this notation, we have
\begin{align}
N(a,b;p)\notag&=\sum\limits_{\lambda\in\F_p\setminus\{0,-1\}}\left(\chi_{I'}\left(\frac{\theta_\lambda}{\pi}\right)+\chi_{I'}\left(\frac{\pi-\theta_\lambda}{\pi}\right)\right)\\ 
\notag&\leq \sum\limits_{\lambda\in\F_p\setminus\{0,-1\}}\left(\sum\limits_{|m|\leq M}\hat{S}_M^{+}(m) \left[e\left(\frac{ m\theta_\lambda}{\pi}\right)+e\left(\frac{- m\theta_\lambda}{\pi}\right)\right]\right) \\
\notag&=\sum\limits_{\lambda\in\F_p\setminus\{0,-1\}}\sum\limits_{|m|\leq M}\hat{S}_M^{+}(m)\cdot 2\cos(2m\theta_\lambda)\\
\notag&=2\hat{S}_M^+(0)\sum\limits_{\lambda\in\F_p\setminus\{0,-1\}}1+\sum\limits_{0<|m|\leq M}\hat{S}^+_M(m)\sum\limits_{\lambda\in\F_p\setminus\{0,-1\}}\left[U_{2m}(\cos\theta_\lambda)-U_{2m-2}(\cos\theta_\lambda)\right].
\end{align}
We regroup terms to get that
\begin{align}
N(a,b;p)\notag&\leq p\left(2\hat{S}_M^+(0)-\left(\hat{S}_M^+(1)+\hat{S}_M^+(-1)\right)\right)\\
\notag&+\sum\limits_{1\leq m\leq M}\left(\hat{S}^+_M(m)+\hat{S}^+_M(-m)-\left(\hat{S}^+_M(m+1)+\hat{S}^+_M(-m-1)\right)\right)\sum\limits_{\lambda\in\F_p\setminus\{0,-1\}}U_{2m}(\cos\theta_\lambda).
\end{align}
The first term approximates the semicircular measure. Namely,
$$
\left|\left(2\hat{S}_M^+(0)-\left(\hat{S}_M^+(1)+\hat{S}_M^+(-1)\right)\right)-2\mu_{ST}([a,b])\right|\leq\frac{4}{M+1}.
$$
Furthermore, if $1\leq m\leq M,$ then
$$
\left|\hat{S}_M^+(m)+\hat{S}_M^+(-m)\right|\leq \frac{2}{m}+\frac{2}{M+1}<\frac{4}{m}.
$$
Therefore, we have
$$
N(a,b;p)-2p\mu_{ST}(I)\leq \frac{4p}{M+1}+\sum\limits_{1\leq m\leq M}\frac{8}{m}\left|\sum\limits_{\lambda\in\F_p\setminus\{0,-1\}} U_{2m}(\cos\theta_\lambda)\right|.
$$
To find an upper bound for the right-hand side, we begin by rewriting the sums in terms of traces of Clausen elliptic curves. More precisely, if $1\leq m\leq M$, then we write
$$
\sum\limits_{\lambda\in\F_p\setminus\{0,-1\}} U_{2m}(\cos\theta_\lambda)=\sum\limits_{l=0}^m\frac{b(2l;2m)}{(2\sqrt{p})^{2l}}\sum\limits_{\lambda\in\F_p\setminus\{0,-1\}} a^{\Cl}_\lambda(p)^{2l}.
$$
This allows us to rewrite the right-hand side as weighted sums of class numbers. Namely,  Proposition~\ref{ClassNumberProp} implies that
$$
\sum\limits_{\lambda\in\F_p\setminus\{0,-1\}} U_{2m}(\cos\theta_\lambda)=\sum\limits_{l=0}^m\frac{b(2l;2m)}{(2\sqrt{p})^{2l}}\left(\sum\limits_{\substack{0<s<2\sqrt{p} \\ 2\mid s}} \left(H^\star\left(4p-s^2\right)+2H^\star\left(\frac{4p-s^2}{4}\right)\right)s^{2l}-c^+(p,l)\right).
$$
Proposition~\ref{ChebyshevProp} allows us to rewrite these weighted sums as Fourier coefficients of cusp forms. To make this precise, write
$
\pi_{\hol}([\mathcal{H}(\tau),\theta(\tau)]_m)=\sum\limits_{n\geq 1}a_m(n)q_{\tau}^n
$
and 
$
\pi_{\hol}([\mathcal{H}(\tau),\theta(4\tau)]_m)=\sum\limits_{n\geq 1}b_m(n)q_{\tau}^n.
$
  In this notation, Proposition~\ref{ChebyshevProp} and Theorem~\ref{HolProjTh} show that
$$
\sum\limits_{l=0}^m b(2l;2m)\sum\limits_{\substack{0<s<2\sqrt{p} \\ 2\mid s}}H^\star(4p-s^2)\left(\frac{s}{2\sqrt{p}}\right)^{2l}=\frac{1}{2p^{m}}\cdot\frac{1}{{{2m}\choose m}}b_m(4p)-\frac{1}{p^m}
$$
and
$$
\sum\limits_{l=0}^m b(2l;2m)\sum\limits_{\substack{0<s<2\sqrt{p} \\ 2\mid s}} 2H^\star\left(\frac{4p-s^2}{4}\right)\left(\frac{s}{2\sqrt{p}}\right)^{2l}=\frac{1}{p^m}\frac{4^m}{{{2m}\choose m}}a_m(p)-\frac{1}{p^m}.
$$
Furthermore, if $c^+(p,l)\neq 0,$ then we have
\begin{align}
\sum\limits_{l=0}^m\frac{b(2l;2m)}{(2\sqrt{p})^{2l}}c^+(p,l)\notag&=\sum\limits_{l=0}^m b(2l;2m)\cdot\frac{1}{2}\left(\left(\frac{a}{\sqrt{p}}\right)^{2l}+\left(\frac{b}{\sqrt{p}}\right)^{2l}\right)\\ 
\notag&=\frac{1}{2}\left(U_{2m}\left(\frac{a}{\sqrt{p}}\right)+U_{2m}\left(\frac{b}{\sqrt{p}}\right)\right),
\end{align}
where $p=a^2+b^2.$ Since $0\leq \frac{a}{\sqrt{p}},\frac{b}{\sqrt{p}}\leq 1,$ Lemma~\ref{BoundChebyshev} gives us that 
$$
\left|\sum\limits_{l=0}^m\frac{b(2l;2m)}{(2\sqrt{p})^{2l}}c^+(p,l)\right|\leq(2m+1).
$$
We now apply Lemma~\ref{BoundCoefficients} to obtain the upper bound\footnote{The purpose of the last term is to make the argument symmetric.}
$$
\left|\sum\limits_{\lambda\in\F_p\setminus\{0,-1\}} U_{2m}(\cos\theta_\lambda)\right|\leq \frac{4}{3}(m-1)p^{1/2}+(2m+1)+\frac{2}{p^m}.
$$
This gives an explicit upper bound on the error. Namely, we have that
$$
N(a,b;p)-2p\mu_{\ST}(I)\leq\frac{4p}{M+1}+\sum\limits_{1\leq m\leq M}\frac{8}{m}\left(\frac{4}{3}(m-1)p^{1/2}+(2m+1)+\frac{2}{p^m}\right).
$$
Since $M$ is an arbitrary positive integer, we choose $M$ in terms of $p$ so that the upper bound is minimal. To this end, we find a simpler upper bound. The elementary inequality  $\frac{2}{p^m}\leq 0.4$ for all $m\geq 1$ and $p\geq 5$ and an elementary computation give us that
$$
N(a,b;p)-2p\mu_{\ST}(I)\leq\frac{4p}{M+1}+\frac{8}{3}\left(6M+4M\sqrt{p}+4.2H_M\right),
$$
where $H_M$ is the $M$-th harmonic number. Since $H_M\leq \log M+1$ for all positive integers $M,$ we then have that
$$
N(a,b;p)-2p\mu_{\ST}(I)\leq\frac{4p}{M+1}+\frac{8}{3}\left(6M+4M\sqrt{p}+4.2\log M+4.2\right).
$$
We set $M=\lfloor p^{1/4}\rfloor.$ This gives us that
$$
N(a,b;p)-2p\mu_{ST}(I)\leq 4p^{3/4}+\frac{8}{3}\left(6p^{1/4}+4p^{3/4}+1.05\log p+4.2\right).
$$
An elementary numerical computation then gives that
$$
\frac{N(a,b;p)}{p}-2\mu_{\ST}(I)\leq\frac{26.52}{p^{1/4}},
$$
for all $p\geq 5.$ An analogous argument with $S_M^{-}$ in place of $S_M^+$ gives the theorem.
\end{proof}
 
To determine the distribution of $A_{\lambda}$ using Theorem~\ref{K3ClausenCon}, we show that the values $\pm1$ of the Legendre character $\phi(-\lambda)$ are equally distributed as $p\to\infty$ for $|a_{\lambda}^{\Cl}(p)|$ lying in a given interval. Namely, if $M(a,b;p):=\sum\limits_{\lambda\in\F_p\setminus\{0,-1\}}\phi(-\lambda)\chi_{[a,b]}(|a_{\lambda}^{\Cl}(p)|)$ is the cancellation in the interval, then we have the following proposition.

\begin{proposition}\label{EffectiveClausenSym2}
If $0\leq a<b\leq 1$ and $p\geq 5,$ then we have
$$
\left|\frac{M(a,b;p)}{p}\right|\leq\frac{28.89}{p^{1/4}}.
$$
\end{proposition}

\begin{proof}
In the notation of the proof of Proposition~\ref{EffectiveClausenSym}, we similarly have
\begin{align}
\notag&M(a,b;p)\leq -\left(2\hat{S}_M^+(0)-\left(\hat{S}_M^+(1)+\hat{S}_M^+(-1)\right)\right)\\
\notag&+\sum\limits_{1\leq m\leq M}\left(\hat{S}^+_M(m)+\hat{S}^+_M(-m)-\left(\hat{S}^+_M(m+1)+\hat{S}^+_M(-m-1)\right)\right)\sum\limits_{\lambda\in\F_p\setminus\{0,-1\}}\phi(-\lambda)U_{2m}(\cos\theta_\lambda).
\end{align}
Imitating the proof of Proposition~\ref{EffectiveClausenSym} and using part (2) of Proposition~\ref{ClassNumberProp}, we get
$$
\left|\sum\limits_{\lambda\in\F_p\setminus\{0,-1\}}\phi(-\lambda)U_{2m}(\cos\theta_\lambda) \right|\leq 2(m-1)p^{1/2}+(2m+1)+\frac{3}{p^{m}}.
$$
This implies that
$$
M(a,b;p)\leq\frac{4}{M+1}+8\left(2M+2M\sqrt{p}+1.6H_M\right),
$$
where $H_M$ is the $M$-th harmonic number. Choosing $M=\lfloor p^{1/4}\rfloor,$ we get
$$
|M(a,b;p)|\leq \frac{28.89}{p^{1/4}},
$$
for all $p\geq 5,$ proving the proposition.
\end{proof}

Propositions~\ref{EffectiveClausenSym} and \ref{EffectiveClausenSym2} explicitly determine the distribution of absolute values of Clausen traces with a fixed Legendre character value. The following proposition makes this precise.

\begin{proposition}\label{EffectiveClausenSym3}
If $0\leq a<b\leq 1$ and $p\geq 5,$ then
$$
\left|\frac{H^{\pm}(a,b;p)}{p}-\mu_{\ST}([a,b])\right|\leq\frac{27.71}{p^{1/4}},
$$
where 
$$
H^{\pm}(a,b;p):=\#\left\{\lambda\in\F_p\setminus\{0,-1\}\ :\ \left|\frac{a_{\lambda}^{\Cl}(p)}{2\sqrt{p}}\right|\in[a,b], \phi(-\lambda)=\pm1\right\}.
$$
\end{proposition}

\begin{proof}
By definition of $M(a,b;p)$ and $N(a,b;p),$ we have that $M(a,b;p)=H^+(a,b;p)-H^{-}(a,b;p)$ and $N(a,b;p)=H^+(a,b;p)+H^{-}(a,b;p).$  Equivalently, we have that
$$
H^{\pm}(a,b;p)=\frac{1}{2}\left(N(a,b;p)\pm M(a,b;p)\right).
$$
An application of Propositions~\ref{EffectiveClausenSym} and \ref{EffectiveClausenSym2} proves our claim.
\end{proof}

\subsection{$O(3)$ Distribution in Terms of Semicircular Distribution}

Here we write the $O(3)$ distribution in terms of the Sato-Tate distribution. To make this precise, if $-3\leq a<b\leq 3,$ then define
\begin{equation}
\mu_{\Bat}([a,b]):=\frac{1}{4\pi}\int_a^b f(t)dt,
\end{equation}
where $f(t)$ is as in Section~\ref{SectionIntro}. In this notation, we have the following lemma.

\begin{lemma}\label{ProbProp}
The following are true.

\noindent
(1) If $0\leq a<b\leq 1,$ then we have
$$
\mu_{\Bat}([a,b])=\mu_{\ST}\left(\left[\frac{\sqrt{1+a}}{2},\frac{\sqrt{1+b}}{2}\right]\right)+\mu_{\ST}\left(\left[\frac{\sqrt{1-b}}{2},\frac{\sqrt{1-a}}{2}\right]\right).
$$

\noindent
(2) If $0 \leq a\leq 1<b\leq 3,$ then we have
$$
\mu_{\Bat}([a,b])=\mu_{\ST}\left(\left[\frac{\sqrt{1+a}}{2},\frac{\sqrt{1+b}}{2}\right]\right)+\mu_{\ST}\left(\left[0,\frac{\sqrt{1-a}}{2}\right]\right).
$$

\noindent
(3) If $1\leq a<b\leq 3,$ then we have
$$
\mu_{\Bat}([a,b])=\mu_{\ST}\left(\left[\frac{\sqrt{1+a}}{2},\frac{\sqrt{1+b}}{2}\right]\right).
$$
\end{lemma}

\begin{proof}
This is an exercise in integration. We will prove (2). The proof of (1) and (3) are analogous.  By the definition of $\mu_{\ST},$ we have that
\begin{align}
\mu_{\ST}\left(\left[\frac{\sqrt{1+a}}{2},\frac{\sqrt{1+b}}{2}\right]\right)\notag&=\frac{2}{\pi}\int_{\frac{\sqrt{1+a}}{2}}^{\frac{\sqrt{1+b}}{2}} \sqrt{1-x^2}dx \\
\notag&=\frac{1}{2\pi}\int_{\sqrt{1+a}}^{\sqrt{1+b}} \sqrt{4-x^2}dx \\
\notag&=\frac{1}{4\pi}\int_{1+a}^{1+b}\frac{\sqrt{4-x}}{\sqrt{x}}dx \\
\notag&=\frac{1}{4\pi}\int_a^b\frac{\sqrt{3-x}}{\sqrt{1+x}dx} \\
\notag&=\frac{1}{4\pi}\int_a^1\frac{\sqrt{3-t}}{\sqrt{1+t}}dt+\frac{1}{4\pi}\int_1^b\frac{\sqrt{3-t}}{\sqrt{1+t}}dt.
\end{align}
Similarly, we have that
$$
\mu_{\ST}\left(\left[0,\frac{\sqrt{1-a}}{2}\right]\right)=\frac{1}{4\pi}\int_a^1\frac{\sqrt{3+t}}{\sqrt{1-t}}dt.
$$
Adding the two terms concludes the proof.
\end{proof}

\section{Proof of Main Results}\label{SectionProofMainResults}

Here we prove Theorem~\ref{MainTheorem} and Corollary~\ref{MainCorollary}.

\begin{proof}[Proof of Theorem~\ref{MainTheorem}]

For brevity, we only prove the case $0\leq a<b\leq 1.$ We start by determining the distribution of $A_{-1-\frac{1}{\lambda}}(p)$ in terms of the absolute values of Clausen traces with a fixed Legendre character value. More precisely, Theorem~\ref{K3ClausenCon} we have that $A_{-1-\frac{1}{\lambda}}(p)\in[a,b]$ if and only if one of the following holds.

\noindent
(1) $\phi(-\lambda)=1$ and $\left|\frac{a^{\Cl}_\lambda(p)}{2\sqrt{p}}\right|\in\left[\frac{\sqrt{1+a}}{2},\frac{\sqrt{1+b}}{2}\right].$

\noindent
(2) $\phi(-\lambda)=-1$ and $\left|\frac{a^{\Cl}_\lambda(p)}{2\sqrt{p}}\right|\in\left[\frac{\sqrt{1-b}}{2},\frac{\sqrt{1-a}}{2}\right].\\ $
Since $\lambda\to-1-\frac{1}{\lambda}$ is bijective on $\F_p\setminus\{0,-1\},$ this implies that 
$$
\#\left\{\lambda\in\F_p\setminus\{0,-1\}\ : A_\lambda(p)\in[a,b]\right\}=H^+\left(\frac{\sqrt{1+a}}{2},\frac{\sqrt{1+b}}{2};p\right)+H^{-}\left(\frac{\sqrt{1-b}}{2},\frac{\sqrt{1-a}}{2};p\right).
$$
Proposition~\ref{EffectiveClausenSym3} then implies that
\begin{align}
\notag&\left|\frac{\#\left\{\lambda\in\F_p\setminus\{0,-1\}\ : A_{{\lambda}}(p)\in[a,b]\right\}}{p}-\mu_{\ST}\left(\left[\frac{\sqrt{1+a}}{2},\frac{\sqrt{1+b}}{2}\right]\right)-\mu_{\ST}\left(\left[\frac{\sqrt{1-b}}{2},\frac{\sqrt{1-a}}{2}\right]\right)\right|\\ 
\notag&\leq\left|\frac{H^+\left(\frac{\sqrt{1+a}}{2},\frac{\sqrt{1+b}}{2};p\right)}{p}-\mu_{\ST}\left(\left[\frac{\sqrt{1+a}}{2},\frac{\sqrt{1+b}}{2}\right]\right)\right|\\ 
\notag&+\left|\frac{H^{-}\left(\frac{\sqrt{1-b}}{2},\frac{\sqrt{1-a}}{2};p\right)}{p}-\mu_{\ST}\left(\left[\frac{\sqrt{1-b}}{2},\frac{\sqrt{1-a}}{2}\right]\right)\right| \\
\notag&\leq\frac{55.42}{p^{1/4}}.
\end{align}
If $0\leq a\leq 1<b\leq 3$ or $1\leq a<b\leq 3,$ the same upper bound holds and the proofs are analogous. If $a<0<b,$ the claim follows by writing $[a,b]=[a,0]\cup[0,b]$ and using the triangle inequality.
\end{proof}

\begin{remark}
	The proof implies the smaller constant $55.42$ when $0<a<b<3$ or $-3<a<b<0.$
\end{remark}

\begin{proof}[Proof of Corollary~\ref{MainCorollary}]

If $T'>\frac{\sqrt{3}}{4\pi}$ is any real number, we first find $x\in(0,1)$ such that $\frac{\mu_{\Bat}([1-x,1])}{x}>T'.$ To this end, note that if $x\in(0,1),$ then
$$
\frac{\mu_{\Bat}([1-x,1])}{x}=\frac{1}{4\pi x}\int_{1-x}^1 f(t)dt\geq\frac{1}{4\pi}\cdot f(1-x),
$$
and therefore, it suffices to find $x\in(0,1)$ such that $f(1-x)>4\pi T'.$ An elementary computation shows that $x=\frac{4}{1+16\pi^2T'^2}$ gives us $f(1-x)>4\pi T'.$ Now, set $T'=T+\delta$ where $\delta>0.$ In this case, the proof of Theorem~\ref{MainTheorem} gives us that
$$
\frac{1}{x}\cdot\frac{\#\left\{\lambda\in\F_{p} : A_\lambda(p)\in [1-x,1]\right\}}{p}\geq T+\delta-\frac{55.42}{xp^{1/4}},
$$
for $x=\frac{4}{1+16\pi^2(T+\delta)^2}.$ The corollary follows by a straightforward computation.
\end{proof}

\end{document}